\documentclass[12pt,a4paper]{amsart}
\usepackage{amsmath}
\usepackage{drpack}
\usepackage[english]{babel}
\usepackage{verbatim}
\usepackage[shortlabels]{enumitem}
\usepackage{color}
\usepackage{tikz-cd}
\usepackage{extarrows}

\newtheoremstyle{mio}%
	{}{} 
	{\itshape}{} 
	{\bfseries}{.}{ } 
	{#1 #2\thmnote{~\mdseries(#3)}} 
\theoremstyle{mio}
\newtheorem{teor}{Theorem}[section]
\newtheorem{cor}[teor]{Corollary}
\newtheorem{prop}[teor]{Proposition}
\newtheorem{lemma}[teor]{Lemma}
\newtheorem{defin}[teor]{Definition}

\newtheoremstyle{definition2}%
	{}{} 
	{}{} 
	{\bfseries}{.}{ } 
	{#1 #2\thmnote{\mdseries~ #3}} 
\theoremstyle{definition2}
\newtheorem{ex}[teor]{Example}
\newtheorem{oss}[teor]{Remark}

\newcommand{\lunghezzegen}{\mathcal{L}}
\newcommand{\lunghezze}{\mathcal{L}_\infty}
\newcommand{\singular}{\lunghezzegen_{\mathrm{sing}}}

\newcommand{\discrete}{\lunghezzegen_{\mathrm{disc}}}
\newcommand{\Specell}{\Sigma}

\newcommand{\locsist}{\mathrm{LocSist}}

\newcommand{\imm}{\mathrm{Im}}
\newcommand{\catmod}{\mathrm{Mod}}
\newcommand{\rank}{\mathrm{rank}}
\newcommand{\rk}{\mathrm{rk}}
\newcommand{\Tor}{\mathrm{Tor}}
\newcommand{\Ann}{\mathrm{Ann}}

\newcommand{\insid}{\mathcal{I}}
\newcommand{\length}{\mathrm{length}}
\newcommand{\inssubmod}{\mathbf{F}}
\newcommand{\inssemistab}{\mathrm{SStar_{st}}}
\newcommand{\qspec}[1]{\mathrm{QSpec}^{#1}}
\newcommand{\psspec}[1]{\mathrm{PsSpec}^{#1}}
\newcommand{\V}{\mathcal{V}}

\title{Decomposition and classification of length functions}
\author{Dario Spirito}
\date{\today}
\email{spirito@mat.uniroma3.it}
\address{Dipartimento di Matematica e Fisica, Universit\`a degli Studi ``Roma Tre'', Roma, Italy}

\begin{document}
\begin{abstract}
We study decompositions of length functions on integral domains as sums of length functions constructed from overrings. We find a standard representation when the integral domain admits a Jaffard family, when it is Noetherian and when it is a Pr\"ufer domains such that every ideal has only finitely many minimal primes. We also show that there is a natural bijective correspondence between singular length functions and localizing systems.
\end{abstract}

\maketitle

\section{Introduction}
The concept of a (generalized) length function on the category $\catmod(R)$ of modules over a ring $R$ was introduced by Northcott and Reufel \cite{northcott_length} as a generalization of the classical length of a module: more precisely, they defined a \emph{length function} as a map from $\catmod(R)$ to the set of nonnegative real numbers (plus infinity) that is additive on exact sequences and such that the length of a module is the supremum of the length of its finitely generated submodules. In particular, they were interested in classifying all the possible length functions on a valuation domain; their results were later deepened and expressed in a different form by Zanardo \cite{zanardo_length}. Shortly after \cite{northcott_length}, V\'amos \cite{vamos-additive} distinguished the two properties used to define a length functions (which he called \emph{additivity} and \emph{upper continuity}), showed that they were independent from each other, and classified all length function on Noetherian rings. Ribenboim \cite{ribenboim-length} subsequently considered length functions with values in an arbitrary ordered abelian group, though he needed to restrict the definition to a smaller class of modules (constructible modules) due to the possible non-completeness of the group (more precisely, due to the possible lack of suprema). More recently, length functions have been linked to the concept of \emph{algebraic entropy} \cite{length-entropy,length-entropy-2,entropy-category}, the study of which also involves invariants satisfying a weaker form of additivity \cite{salce-zanardo-forum}.

The purpose of this paper is to investigate two closely related problems: the first is the possibility of ``decomposing'' a length function $\ell$ on an integral domain $D$ as a sum of length functions defined on overrings of $D$ (in particular, localizations of $D$); the second is the possibility of expressing the set $\lunghezzegen(D)$ of length functions on $D$ (and/or some distinguished subset) as a product of the set of length functions on a family of overrings of $D$. Both problems can be seen, more generally, as asking for a way to find all length functions on $D$ by reducing to ``simpler'' domains and cases.

In Section \ref{sect:jaffard} (which can be seen as a generalization the case of Dedekind domains treated in \cite[Section 7]{northcott_length}), we start by analyzing ways to construct a length function on an overring of $D$ from a length function on $D$ and, conversely, to construct a length function on $D$ from a length function on an overring (or, more generally, from a family of overrings) and how these two operations relate one to each other. We show that the best results are obtained when we consider a \emph{Jaffard family} of $D$, i.e., a family $\Theta$ of flat overrings of $D$ that is complete, independent and locally finite (see Section \ref{sect:background} for a precise definition): in particular, we show that every length function $\ell$ is equal to the sum $\sum_{T\in\Theta}\ell\otimes T$ (where $\ell\otimes T$ sends a module $M$ to $\ell(M\otimes_DT)$; Theorem \ref{teor:jaff-scompo}) and that the set $\lunghezze(D)$ of length functions such that $\ell(D)=\infty$ is order-isomorphic to the product $\prod_{T\in\Theta}\lunghezze(T)$.

In Section \ref{sect:primary}, we study primary ideals; in particular, using V\'amos' results, we show that any length function $\ell$ on a Noetherian domain $D$ can be written as $\sum_{P\in\Sigma(\ell)}\ell\otimes D_P$ (where $\Sigma(\ell)$ is a subset of $\Spec(D)$ depending on $\ell$; Proposition \ref{prop:noeth}) and that if $I$ is an ideal (of an arbitrary domain $D$) with a primary decomposition, then $\ell(D/I)$ is the sum of $\ell(D/Q)$, as $Q$ ranges among the primary components of $I$ (Proposition \ref{prop:primary}).

In Section \ref{sect:prufer}, we study length function on Pr\"ufer domains, in particular on those such that every ideal has only finitely many minimal primes. As in the Noetherian case, we show that we can always write $\ell=\sum_{P\in\Sigma(\ell)}\ell\otimes D_P$ (Theorem \ref{teor:prufer}) and we use this representation to prove that, for these class of domains, the set $\lunghezzegen(D)$ depends only on the topological structure of $\Spec(D)$ and on which prime ideals are idempotent (Theorem \ref{teor:corrisp}).

In Section \ref{sect:singular}, we characterize singular length functions (i.e., length functions such that $\ell(M)$ can only be $0$ or $\infty$) on an integral domain $D$ by finding a natural bijection between their set and the set of localizing systems on $D$ (Theorem \ref{teor:singular}). We also consider the relationship between singular length functions and stable semistar operations.

\section{Background and notation}\label{sect:background}
Let $\insR^{\geq 0}$ denote the set of nonnegative real numbers, and let $\Gamma:=\insR^{\geq 0}\cup\{\infty\}$. Then, $\Gamma$ has a natural structure of (commutative) ordered additive semigroup, where, for every $r\in\Gamma$, $r+\infty=\infty+r=\infty$ and $r\leq\infty$.

If $\Lambda\subseteq\Gamma$ is a (not necessarily finite) set, we define the sum of $\Lambda$ as the supremum of all the finite sums $\lambda_1+\cdots+\lambda_n$, as $\{\lambda_1,\ldots,\lambda_n\}$ ranges among the finite subset of $\Lambda$. Since all elements of $\Lambda$ are nonnegative, this notion coincides with the usual sum if $\Lambda$ is finite.

\medskip

Let $R$ be a ring, and let $\catmod(R)$ be the category of $R$-modules. A map $\ell:\catmod(R)\longrightarrow\Gamma$ is a \emph{length function} on $R$ if:
\begin{itemize}
\item $\ell(0)=0$;
\item $\ell$ is \emph{additive}: for every short exact sequence
\begin{equation*}
0\longrightarrow M_1\longrightarrow M_2\longrightarrow M_3\longrightarrow 0,
\end{equation*}
we have $\ell(M_2)=\ell(M_1)+\ell(M_3)$;
\item $\ell$ is \emph{upper continuous}: for every $R$-module $M$,
\begin{equation*}
\ell(M)=\sup\{\ell(N)\mid N\text{~is a finitely generated submodule of~}M\}.
\end{equation*}
\end{itemize}

Note that the first hypothesis is adopted by Northcott and Reufel \cite{northcott_length} and by V\'amos \cite{vamos-additive}, but not by Zanardo \cite{zanardo_length}. The only function it excludes is the map $\ell_\infty$ sending every module to $\infty$. We call the length function $\ell_0$ such that $\ell_0(M)=0$ for every module $M$ the \emph{zero length function}; in \cite[Section 3.1]{zanardo_length}, $\ell_0$ and $\ell_\infty$ are called \emph{trivial} length function, but we shall not use this terminology.

If $\ell$ is different from the zero length function, then $\ell(R)>0$.

It is easily seen from the definition that if $M_1$ and $M_2$ are isomorphic then $\ell(M_1)=\ell(M_2)$, and that if $N$ is a submodule or a quotient of $M$ then $\ell(N)\leq\ell(M)$.

Three examples of length functions on a ring $R$ are:
\begin{itemize}
\item the ``usual'' length function (i.e., the Jordan-H\"older length of a module): we denote it by $\length_R$;
\item the function $\ell$ such that $\ell(M)=0$ if $M$ is a torsion $R$-module, while $\ell(M)=\infty$ otherwise;
\item if $R$ is an integral domain, the rank function: $\rank(M):=\dim_K(M\otimes K)$, where $K$ is the quotient field of $R$.
\end{itemize}
If $R$ is an integral domain, then the rank function is, up to multiplication by a constant, the only length function $\ell$ such that $\ell(R)<\infty$ \cite[Theorem 2]{northcott_length}.

Let $\imm(\ell)$ denote the image of $\ell$, i.e., the set of $\ell(M)$ as $M$ ranges in $\catmod(R)$. We say that a length function $\ell$ is:
\begin{itemize}
\item \emph{singular} if $\imm(\ell)=\{0,\infty\}$;
\item \emph{discrete} if $\imm(\ell)$ is discrete in $\Gamma$.
\end{itemize}
We denote by:
\begin{itemize}
\item $\lunghezzegen(R)$ the set of length functions on $R$;
\item $\lunghezze(R)$ the set of length functions such that $\ell(R)=\infty$;
\item $\singular(R)$ the set of singular length functions;
\item $\discrete(R)$ the set of discrete length functions.
\end{itemize}

The set $\lunghezzegen(R)$ has a natural order structure, where $\ell_1\leq\ell_2$ if and only if $\ell_1(M)\leq\ell_2(M)$ for every $R$-module $M$. In this order, $\lunghezzegen(R)$ has both a minimum ($\ell_0$) and a maximum (the function sending all non-zero modules to $\infty$).

We shall often use the following result.
\begin{prop}
\cite[Proposition 3.3]{zanardo_length} Let $\ell_1,\ell_2$ be length functions of $R$. If $\ell_1(R/I)=\ell_2(R/I)$ for every ideal $I$ of $R$, then $\ell_1=\ell_2$.
\end{prop}

\medskip

Let $D$ be an integral domain with quotient field $K$; an \emph{overring} of $D$ is a ring comprised between $D$ and $K$. A \emph{Jaffard family} of $D$ is a family $\Theta$ of overrings of $D$ such that \cite[Proposition 4.3]{starloc}:
\begin{itemize}
\item every $T\in\Theta$ is flat;
\item $K\notin \Theta$;
\item $I=\bigcap\{IT\mid T\in\Theta\}$ for every ideal $I$ of $D$ (i.e., $\Theta$ is \emph{complete});
\item $TS=K$ for every $T\neq S$ belonging to $\Theta$ (i.e., $\Theta$ is \emph{independent});
\item for every $x\in K$, there are only finitely many $T\in\Theta$ such that $x$ is not a unit in $T$ (i.e., $\Theta$ is \emph{locally finite}).
\end{itemize}
Note that this is not the original definition; see \cite[beginning of Section 6.3 and Theorem 6.3.5]{fontana_factoring} for two other characterizations. 

In particular, if $\Theta$ is a Jaffard family of $D$ and $P\in\Spec(D)$ is nonzero, then there is exactly one $T\in\Theta$ such that $PT\neq T$ \cite[Theorem 6.3.1(1)]{fontana_factoring}.

\medskip

A \emph{Pr\"ufer domain} is an integral domain such that the localization at every prime ideal is a valuation domain. If $V$ is a valuation domain and $P\in\Spec(V)$, we say that $P$ is \emph{branched} if $P$ is minimal over a principal ideal; equivalently, if the union of all the prime ideals properly contained in $P$ is different from $P$. A prime ideal that is not branched is called \emph{unbranched}. If $D$ is a Pr\"ufer domain, we say that $P\in\Spec(D)$ is branched if $PD_P$ is branched in $D_P$.

\medskip

All rings considered are commutative and unitary. All unreferenced results on Pr\"ufer and valuation domains are standard; see for example \cite{gilmer} for a general reference.

\section{Jaffard families}\label{sect:jaffard}
The following section is a generalization of \cite[Section 7]{northcott_length}, of which we follow and generalize the method.

Let $D$ be an integral domain and $T$ be a $D$-algebra. Then, every $T$-module is also (in a canonical way) a $D$-module; thus, given any $\ell\in\lunghezzegen(D)$, we can define a function $\ell_T$ by
\begin{equation*}
\ell_T(M):=\ell(M)\quad\text{for every~}M\in\catmod(T).
\end{equation*}

\begin{prop}\label{prop:loclength}
Let $D,T,\ell$ as above.
\begin{enumerate}[(a)]
\item\label{prop:loclength:length} $\ell_T$ is a length function on $T$.
\item\label{prop:loclength:sing} If $\ell$ is singular, so is $\ell_T$.
\item\label{prop:loclength:disc} If $\ell$ is discrete, so is $\ell_T$.
\end{enumerate}
\end{prop}
\begin{proof}
\ref{prop:loclength:length} Let $0\longrightarrow M_1\longrightarrow M_2\longrightarrow M_3\longrightarrow 0$ be an exact sequence of $T$-modules. Then, it is also an exact sequence of $D$-modules; hence,
\begin{equation*}
\ell_T(M_2)=\ell(M_2)=\ell(M_1)+\ell(M_3)=\ell_T(M_1)+\ell_T(M_3).
\end{equation*}
Thus, $\ell_T$ is additive. Suppose now $\ell_T(M)>x$ for some $x\inR$. Then, $\ell(M)>x$, and thus there is a finitely generated $D$-submodule $N$ of $M$ such that $\ell(N)>x$; let $N:=e_1D+\cdots+e_kD$. Then, $NT=e_1T+\cdots+e_nT$ is a submodule of $M$ containing $N$, and thus
\begin{equation*}
\ell_T(NT)=\ell(NT)\geq\ell(N)>x.
\end{equation*}
Hence, $\ell_T(M)=\sup\{\ell_T(N')\mid N'\subseteq M\text{~is finitely generated over~}T\}$, and thus $\ell_T$ is upper continuous. Therefore, $\ell_T$ is a length function on $T$.

\ref{prop:loclength:sing} and \ref{prop:loclength:disc} are obvious.
\end{proof}

In general, it is possible for $\ell_T$ to be the zero length function even if $\ell$ is not:  for example, if $\ell(D/I)=0$ and $T=D/I$, then $\ell_T(M)$ will be $0$ for all $T$-modules $M$.
\begin{prop}\label{prop:loclength-overring}
Let $D$ be an integral domain and $T$ be a $D$-algebra.
\begin{enumerate}[(a)]
\item\label{prop:loclength-overring:triv} If $T$ is torsion-free over $D$, then $\ell(T)\neq 0$ for all nonzero $\ell\in\lunghezzegen(D)$.
\item\label{prop:loclength-overring:inf} If $T$ is torsion-free over $D$ and $\ell\in\lunghezze(D)$, then $\ell_T\in\lunghezze(T)$.
\item\label{prop:loclength-overring:rank} If $T$ is torsion-free over $D$ and $\rank_D(T)<\infty$, then $\ell\in\lunghezze(D)$ if and only if $\ell_T\in\lunghezze(T)$.
\item\label{prop:loclength-overring:over} If $T$ is an overring of $D$, then $\ell\in\lunghezze(D)$ if and only if $\ell_T\in\lunghezze(T)$.
\end{enumerate}
\end{prop}
\begin{proof}
\ref{prop:loclength-overring:triv} If $T$ is torsion-free, then the canonical map $D\longrightarrow T$ is injective; hence, it $\ell$ is an arbitrary nonzero length function, $\ell_T(T)=\ell(T)\geq\ell(D)>0$.

\ref{prop:loclength-overring:inf} As above, for every $\ell\in\lunghezze(D)$ we have $\ell_T(T)=\ell(T)\geq\ell(D)=\infty$, and thus $\ell_T\in\lunghezze(T)$.

\ref{prop:loclength-overring:rank} Suppose $T$ is torsion-free and $\rank_D(T)<\infty$, and let $\ell\in\lunghezzegen(D)$. If $\ell\in\lunghezze(D)$ then $\ell_T\in\lunghezze(T)$ by the previous point. On the other hand, if $\ell\notin\lunghezze(D)$, then by \cite[Theorem 2]{northcott_length}, $\ell(M)=\alpha\cdot\rank_D(M)$ for every $D$-module $M$, where $\alpha:=\ell(D)$; in particular, $\ell_T(T)=\ell(T)=\rank_D(T)<\infty$ by hypothesis, and thus $\ell_T\notin\lunghezze(T)$.

\ref{prop:loclength-overring:over} follows from the previous point, since $T\otimes_DK\simeq D$ (where $K$ is the quotient field of $D$) and so $\rank_D(T)=1$.
\end{proof}

Hence, if $T$ is a torsion-free $D$-algebra then by the previous proposition we have a map
\begin{equation*}
\begin{aligned}
\widehat{\Lambda}_{D,T}\colon\lunghezzegen(D) & \longrightarrow\lunghezzegen(T)\\
\ell & \longmapsto \ell_T,
\end{aligned}
\end{equation*}
that restricts to a map
\begin{equation*}
\begin{aligned}
\Lambda_{D,T}\colon\lunghezze(D) & \longrightarrow\lunghezze(T)\\
\ell & \longmapsto \ell_T.
\end{aligned}
\end{equation*}
Clearly, both $\widehat{\Lambda}_{D,T}$ and $\Lambda_{D,T}$ are order-preserving.

\medskip

A more interesting question is if (and how) we can construct a length function on $D$ from a length function on a $D$-algebra $T$. We shall mainly be interested in the case when $T=D_P$ is a localization of $D$, but there is no harm in working more generally with a flat algebra.

Let thus $T$ be a flat $D$-algebra, and let $\ell\in\lunghezzegen(T)$. We define $\ell^D$ as the map such that
\begin{equation*}
\ell^D(M):=\ell(T\otimes_DM)\quad\text{for all~}M\in\catmod(D).
\end{equation*}
This construction behaves similarly to the construction $\ell_T$.
\begin{prop}\label{prop:otimeslength}
Let $D,T,\ell$ as above.
\begin{enumerate}[(a)]
\item\label{prop:otimeslength:length} $\ell^D\in\lunghezzegen(D)$.
\item\label{prop:otimeslength:imm} $\ell^D\in\lunghezze(D)$ if and only if $\ell\in\lunghezze(T)$.
\end{enumerate}
\end{prop}
\begin{proof}
\ref{prop:otimeslength:length} follows in the same way of \cite[Proposition 2]{northcott_length}, using the flatness of $T$. \ref{prop:otimeslength:imm} is immediate because $\ell^D(D)=\ell(T\otimes_DD)=\ell(T)$.
\end{proof}

Therefore, we can define a map $\widehat{\Psi}_{T,D}$ by setting
\begin{equation*}
\begin{aligned}
\widehat{\Psi}_{T,D}\colon\lunghezzegen(T) & \longrightarrow\lunghezzegen(D)\\
\ell & \longmapsto \ell^D;
\end{aligned}
\end{equation*}
by part \ref{prop:otimeslength:imm} of the previous proposition also its restriction
\begin{equation*}
\begin{aligned}
\Psi_{T,D}\colon\lunghezze(T) & \longrightarrow\lunghezze(D)\\
\ell & \longmapsto \ell^D
\end{aligned}
\end{equation*}
is well-defined.

\begin{prop}
Let $D$ be an integral domain and $T$ be a flat overring of $D$. Then, $\widehat{\Lambda}_{D,T}\circ\widehat{\Psi}_{T,D}$ is the identity on $\lunghezzegen(T)$.
\end{prop}
\begin{proof}
Let $\ell\in\lunghezzegen(T)$. Then, for every $M\in\catmod(T)$, we have
\begin{equation*}
(\widehat{\Lambda}_{D,T}\circ\widehat{\Psi}_{T,D})(\ell)(M)=(\ell^D)_T(M)=\ell^D(M)= \ell(M\otimes_DT).
\end{equation*}
Since $T$ is a flat overring of $D$, the inclusion $D\hookrightarrow T$ is an epimorphism (being $D\hookrightarrow K$ an epimorphism; see \cite[Chapitre IV, Corollaire 3.2]{lazard_flat} or \cite[Proposition 4.5]{knebush-zhang}); hence, for every $T$-module $D$, we have
\begin{equation*}
T\otimes_DM\simeq T\otimes_D(T\otimes_T M)\simeq(T\otimes_D T)\otimes_TM\simeq T\otimes_TM\simeq M
\end{equation*}
as $T$-modules, with the second-to-last equality coming from the fact that the inclusion is an epimorphism (see \cite[Lemma 1.0]{lazard_flat} or \cite[Lemma A.1]{knebush-zhang}). Thus, $(\ell^D)_T(M)=\ell(M)$; since $M$ was arbitrary, $(\ell^D)_T=\ell$, i.e., $\widehat{\Lambda}_{D,T}\circ\widehat{\Psi}_{T,D}$ is the identity.
\end{proof}

\begin{oss}
The previous proposition does not work for arbitrary flat $D$-algebras. For example, if $T=M=D[x]$ is the polynomial ring over $D$, then $M\otimes_DT\simeq D[x,y]=T[y]$ as $T$-modules; hence, if $\ell$ is the rank function of $T$, we have $\ell(M)=1$ while $(\ell^D)_T(M)=\infty$.
\end{oss}

Let now $\ell$ be a length function on $D$, and let $T$ be a flat $D$-algebra. We set
\begin{equation*}
\ell\otimes T:=(\widehat{\Psi}_{T,D}\circ\widehat{\Lambda}_{D,T})(\ell).
\end{equation*}
This notation comes from the fact that, if $M$ is a $D$-module, then
\begin{equation*}
(\ell\otimes T)(M)=(\ell_T)^D(M)=\ell_T(M\otimes_DT)=\ell(M\otimes_DT).
\end{equation*}
The construction $\ell\otimes T$ is easily seen to satisfy an associative-like property: if $T_1,T_2$ are flat $D$-algebras, then
\begin{equation*}
(\ell\otimes T_1)\otimes T_2=\ell\otimes(T_1\otimes_DT_2).
\end{equation*}

In general, $\ell\otimes T$ is different from $\ell$; for example, if $I$ is a proper ideal of $D$ such that $IT=T$, then
\begin{equation*}
(\ell\otimes T)(D/I)=\ell(D/I\otimes_DT)=\ell(T/IT)=\ell(0)=0,
\end{equation*}
but clearly there may be length functions such that $\ell(D/I)\neq 0$.

Hence, if we want to construct an isomorphism from $\widehat{\Lambda}_{D,T}$ and $\widehat{\Psi}_{T,D}$, we need to consider more rings, and to do so we must define the sum of a family of length functions.
\begin{defin}
Let $\Lambda:=(\ell_\alpha)_{\alpha\in A}$ be a family of length functions on $D$. The \emph{sum} of $\Lambda$, which we denote by $\sum_{\alpha\in A}\ell_\alpha$, is the map 
\begin{equation*}
\begin{aligned}
\sum_{\alpha\in A}\ell_\alpha\colon\catmod(D) & \longrightarrow\Gamma\\
M & \longmapsto \sum_{\alpha\in A}\ell_\alpha(M).
\end{aligned}
\end{equation*}
\end{defin}
Recall that the sum of an arbitrary family of elements of $\Gamma$ is defined as the supremum of the set of the finite sums (see the beginning of Section \ref{sect:background}).

\begin{lemma}\label{lemma:sommalungh}
The sum of a family $\Lambda$ of length functions on $D$ is a length function, and is the supremum of $\Lambda$ in $\lunghezzegen(D)$.
\end{lemma}
\begin{proof}
Just apply the definitions.
\end{proof}

\begin{lemma}\label{lemma:directsum}
Let $\{M_\alpha\}_{\alpha\in A}$ be a family of $D$-modules, and let $\ell$ be a length function on $D$. Then,
\begin{equation*}
\ell\left(\bigoplus_{\alpha\in A}M_\alpha\right)=\sum_{\alpha\in A}\ell(M_\alpha).
\end{equation*}
\end{lemma}
\begin{proof}
If the family is finite, the claim follows immediately from additivity and induction on the cardinality of $A$.
	
If $A$ is an arbitrary family, set $N$ to be the direct product of the $M_\alpha$; then, $\ell(N)\geq\ell(M_{\alpha_1})+\cdots+\ell(M_{\alpha_n})$ for all finite subsets $\alpha_1,\ldots,\alpha_n$. Furthermore, every finitely generated submodule of $N$ is contained in $M_{\alpha_1}\oplus\cdots\oplus M_{\alpha_n}$ for some $\alpha_1,\ldots,\alpha_n$; by upper continuity, it follows that $\ell(N)$ must be exactly the supremum. The claim is proved.
\end{proof}

The following lemma is a (partial) generalization of a property of $h$-local domains; see \cite[Theorem 22]{matlis-tfm}.
\begin{lemma}\label{lemma:jaffard-torsion}
Let $D$ be an integral domain and $\Theta$ be a Jaffard family of $D$. If $M$ is a torsion $D$-module, then
\begin{equation*}
M\simeq\bigoplus_{T\in\Theta}M\otimes_DT.
\end{equation*}
In particular, if $I\neq(0)$ is an ideal of $D$, then
\begin{equation*}
\frac{D}{I}\simeq\bigoplus_{T\in\Theta}\frac{T}{IT}\simeq\bigoplus_{T\in\Theta}\left(\frac{D}{I}\otimes_DT\right).
\end{equation*}
\end{lemma}
\begin{proof}
We shall follow the proof of \cite[Theorem 22, $3\Longrightarrow 4$]{matlis-tfm}; we start by showing that $K/D\simeq\bigoplus_{T\in\Theta}K/T$. Indeed, consider the natural map
\begin{equation*}
\begin{aligned}
\Phi\colon K & \longrightarrow\bigoplus_{T\in\Theta}K/T \\
d & \longmapsto d+T.
\end{aligned}
\end{equation*}
Note that $\Phi$ is well-defined since $\Theta$ is locally finite.

The kernel of $\Phi$ is $\bigcap_{T\in\Theta}T=D$. Hence, we need only to show that $\Phi$ is surjective, and to do so it is enough to show that every element of the form $e(\alpha,U):=(0,\ldots,0,\alpha+U,0,\ldots,0)$ is in the image of $\Phi$; i.e., we need to show that for every $U\in\Theta$ and every $\alpha\in K$ there is an $\alpha'\in K$ such that $\alpha'-\alpha\in U$ while $\alpha'\in T$ for every $T\in\Theta\setminus U$.

Let $U':=\bigcap_{T\in\Theta\setminus U}T$; by \cite[Proposition 4.5(b)]{starloc}, $U'U=K$. Hence, we have
\begin{equation*}
U+U'=\bigcap_{T\in\Theta}(U+U')T=(U+UU')\cap\bigcap_{T\in\Theta\setminus U}(UT+U'T)=K.
\end{equation*}
In particular, $\alpha=\beta+\alpha'$ for some $\beta\in U$, $\alpha'\in U'$; since $\alpha'-\alpha=\beta\in U$, we have $e(\alpha,U)=\Phi(\alpha')$, as required. Hence, $\Phi$ is surjective and $K/D\simeq\bigoplus_{T\in\Theta} K/T$.

Let now $M$ be a torsion $D$-module; then, since every $T$ is flat we have (see e.g. \cite[p.9-10]{matlis-tfm})
\begin{equation*}
\begin{aligned}
M &\simeq\Tor_1^D(K/D,M)\simeq\Tor_1^D\left(\bigoplus_{T\in\Theta}K/T,M\right)\simeq\\
& \simeq\bigoplus_{T\in\Theta}\Tor_1^T(K/T,M)\simeq\bigoplus_{T\in\Theta}M\otimes_DT,
\end{aligned}
\end{equation*}
as claimed. The ``in particular'' statement follows from the fact that $D/I\otimes_DT\simeq T/IT$.
\end{proof}

\begin{teor}\label{teor:jaff-scompo}
Let $D$ be an integral domain, and let $\Theta$ be a Jaffard family of $D$. For every length function $\ell\in\lunghezze(D)$, we have
\begin{equation*}
\ell=\sum_{T\in\Theta}\ell\otimes T.
\end{equation*}
\end{teor}
\begin{proof}
Let $\ell^\sharp:=\sum_{T\in\Theta}\ell\otimes T$; by Lemma \ref{lemma:sommalungh}, $\ell^\sharp$ is a length function on $D$. To show that $\ell=\ell^\sharp$, it is enough to show that $\ell(D/I)=\ell^\sharp(D/I)$ for every ideal $I$ of $D$.

If $I=(0)$ then $\ell(D/I)=\infty=\ell^\sharp(D/I)$. Suppose $I\neq(0)$. By Lemmas \ref{lemma:directsum} and \ref{lemma:jaffard-torsion}, we have
\begin{equation*}
\ell(D/I)=\ell\left(\bigoplus_{T\in\Theta}\frac{D}{I}\otimes T\right)= \sum_{T\in\Theta}\ell\left(\frac{D}{I}\otimes T\right)= \sum_{T\in\Theta}(\ell\otimes T)(D/I)=\ell^\sharp(D/I).
\end{equation*}
The claim is proved.
\end{proof}

\begin{teor}\label{teor:jaffard}
Let $\Theta$ be a Jaffard family of $D$, and let $\Lambda_\Theta$ and $\Psi_\Theta$ be the maps
\begin{equation*}
\begin{aligned}
\Lambda_\Theta\colon\lunghezze(D) & \longrightarrow\prod_{T\in\Theta}\lunghezze(T)\\
\ell & \longmapsto (\ell_T)_{T\in\Theta}
\end{aligned}\quad\text{and}\quad\begin{aligned}
\Psi_\Theta\colon\prod_{T\in\Theta}\lunghezze(T) & \longrightarrow\lunghezze(D)\\
(\ell_{(T)})_{T\in\Theta} & \longmapsto \sum_{T\in\Theta}(\ell_{(T)})^D.
\end{aligned}
\end{equation*}
Then, the following hold.
\begin{enumerate}[(a)]
\item\label{teor:jaffard:lungh} $\Lambda_\Theta$ and $\Psi_\Theta$ are order-preserving bijections between $\lunghezze(D)$ and $\prod_{T\in\Theta}\lunghezze(T)$, inverse one of each other.
\item\label{teor:jaffard:sing} $\Lambda_\Theta$ restricts to a bijection from $\singular(D)$ to $\prod_{T\in\Theta}\singular(T)$.
\item\label{teor:jaffard:discrete} If $\Theta$ is finite, $\Lambda_\Theta$ restricts to a bijection from $\discrete(D)$ to $\prod_{T\in\Theta}\discrete(T)$.
\end{enumerate}
\end{teor}
\begin{proof}
\ref{teor:jaffard:lungh} By Propositions \ref{prop:loclength}, \ref{prop:otimeslength} and Lemma \ref{lemma:sommalungh}, $\Lambda_\Theta$ and $\Psi_\Theta$ are well-defined.

By definition,
\begin{equation*}
\Psi_\Theta\circ\Lambda_\Theta(\ell)=\sum_{T\in\Theta}(\ell_T)^D=\sum_{T\in\Theta}\ell\otimes T,
\end{equation*}
which is equal to $\ell$ by Theorem \ref{teor:jaff-scompo}. Hence, $\Psi_\Theta\circ\Lambda_\Theta$ is the identity on $\lunghezzegen(D)$.

Take now $(\ell_{(T)})_{T\in\Theta}\in\prod_{T\in\Theta}\lunghezze(T)$. Fix $U\in\Theta$, and let $\ell'$ be the component with respect to $U$ of $(\Lambda_\Theta\circ\Psi_\Theta)(\ell^T)$. As in the proof of Theorem \ref{teor:jaff-scompo}, we need to show that $\ell_{(U)}(U/J)=\ell'(U/J)$ for every ideal $J$ of $U$. If $J=(0)$ then both sides are infinite; suppose $J\neq(0)$. Then,
\begin{equation*}
\ell'(U/J)=\ell_{(U)}\left(\sum_{T\in\Theta}(U/J)\otimes_DT\right)=\sum_{T\in\Theta}\ell_{(T)}((U/J)\otimes_DT).
\end{equation*}
If $U\neq T$, then
\begin{equation*}
\frac{U}{J}\otimes_DT\simeq\frac{D}{J\cap D}\otimes_DU\otimes_DT\simeq\frac{T}{(J\cap D)T}\otimes U=(0),
\end{equation*}
since $(J\cap D)T=T$; hence, $\ell'(U/J)$ reduces to $\ell_{(U)}((U/J)\otimes_DU)=\ell_{(U)}(U/J)$. Therefore, $\ell'=\ell_{(U)}$, and so $\Lambda_\Theta\circ\Psi_\Theta$ is the identity, as claimed.

\ref{teor:jaffard:sing} If $\ell$ is singular, then so is every $\ell_T$, by Proposition \ref{prop:loclength}\ref{prop:loclength:sing}. Conversely, if every $\ell_T$ is singular, then $\sum_T\ell_T(M\otimes T)$ is always zero or infinite, and thus $\ell$ is singular.

\ref{teor:jaffard:discrete} If $\ell$ is discrete, so is every $\ell_T$, by Proposition \ref{prop:loclength}\ref{prop:loclength:disc}. On the other hand, if $\Theta$ is finite, say $\Theta:=\{T_1,\ldots,T_n\}$, then
\begin{equation*}
\imm(\ell)=\{a_1+\cdots+a_n\mid a_i\in\imm(\ell_{T_i})\}.
\end{equation*}
Since each $\imm(\ell_{T_i})$ is discrete, so is $\imm(\ell)$. The claim is proved.
\end{proof}

The fact that every infinite length function $\ell$ can be ``decomposed'' as a sum of $\ell\otimes T$ is a rather special property, which puts some constraints on the family $\Theta$.
\begin{prop}\label{prop:vicev-jaff}
Let $D$ be an integral domain, and let $\Theta$ be a nonempty family of flat overrings of $D$.
\begin{enumerate}[(a)]
\item\label{prop:vicev-jaff:gen} If, for every $\ell\in\lunghezzegen(D)$, we have $\ell=\sum_{T\in\Theta}\ell\otimes T$, then $\Theta=\{D\}$.
\item\label{prop:vicev-jaff:inf} If, for every $\ell\in\lunghezze(D)$, we have $\ell=\sum_{T\in\Theta}\ell\otimes T$, then for every nonzero prime $P$ of $D$ there is exactly one $T\in\Theta$ such that $PT\neq T$.
\end{enumerate}
\end{prop}
\begin{proof}
Given a length function $\ell$ on $D$, let $\ell^\sharp:=\sum_{T\in\Theta}\ell\otimes T$.

\ref{prop:vicev-jaff:gen} Let $\ell$ be the rank function on $D$. Then, $\ell(D)=\ell(T)=1$ for every overring $T$ of $D$, and in particular $(\ell\otimes T)(D)=1$ for every flat overring $T$. Hence, $\ell^\sharp(D)=\sum_{T\in\Theta}1=|\Theta|$; since by hypothesis $\ell=\ell^\sharp$ we must have $|\Theta|=1$; let $\Theta=\{T\}$.

Suppose $T\neq D$: then, since $T$ is flat there must be a prime ideal $P$ of $D$ such that $PT=T$. Let $\ell:=(\length_{D_P})^D$: then, 
\begin{equation*}
\ell(D/P)=\length_{D_P}(D/P\otimes_DD_P)=\length_{D_P}(D_P/PD_P)=1.
\end{equation*}
On the other hand,
\begin{equation*}
\ell^\sharp(D/P)=(\ell\otimes T)(D/P)=\ell(D/P\otimes T)=\ell(T/PT)=\ell(0)=0,
\end{equation*}
against the hypothesis. Hence, $T$ must be equal to $D$ and $\Theta=\{D\}$, as claimed.

\ref{prop:vicev-jaff:inf} Fix a nonzero prime ideal $P$ of $D$, and let $\Theta':=\{P\in\Theta\mid PT\neq T\}$. Let $\ell:=(\length_{D_P})^D$; then, $\ell\in\lunghezze(D)$, and with the same calculation of the previous point we see that $\ell(D/P)=1$ while $\ell^\sharp(D/P)=|\Theta'|$, which means that $|\Theta'|=1$. The claim is proved.
\end{proof}

\begin{oss}\label{oss:teorjaff-no}
~\begin{enumerate}
\item Under the hypotheses of part \ref{prop:vicev-jaff:inf} of the previous proposition, the only property needed to show that $\Theta$ is a Jaffard family is the local finiteness of $\Theta$. We shall see in Example \ref{ex:ad} that, at least for singular length functions, this is not actually necessary; on the other hand, we shall present in Example \ref{ex:global} the case of a one-dimensional domain having a length function $\ell$ that cannot be written as $\sum\ell\otimes D_M$ (with the sum ranging among the maximal ideals).

\item If $\Theta$ is an infinite Jaffard family, there could be elements in $\prod_T\discrete(T)$ that does not come from discrete length functions. For example, suppose $\Theta=\{T_1,\ldots,T_n,\ldots\}$ is countable and that for every $i$ there is a discrete length function $\ell_i$ on $T_i$ such that $\ell_i(M_i)\neq\{0,\infty\}$ for some torsion $T_i$-module $M_i$. By possibly multiplying for some constant, we can suppose $\ell_i(M_i)=\inv{2^i}$ for every $i$. Let $\ell:=\Psi_\Theta(\ell_i)$. Then, $\ell(M_i)=\inv{2^i}$ for every $i$; hence, we have $\ell\left(\bigoplus_{i\leq k}M_i\right)=1-\inv{2^k}$ and $\ell\left(\bigoplus_{i\inN}M_i\right)=1$. In particular, $\imm(\ell)$ is not discrete, since 1 is not an isolated point of $\imm(\ell)$.
\end{enumerate}
\end{oss}

\begin{ex}
Let $D$ be a one-dimensional locally finite domain. Then, $\{D_M\mid M\in\Max(D)\}$ is a Jaffard family of $D$. By Theorem \ref{teor:jaffard}, it follows that a length function $\ell$ on $D$ is completely determined by the restrictions $\ell_{D_P}$. In turn, this means that $\ell$ is completely determined by its values at $\ell(D_P/QD_P)$, where $Q$ is a $P$-primary ideal; as we shall see in Proposition \ref{prop:primary} below, this means that $\ell$ is determined by the values $\ell(D/Q)$, as $Q$ ranges among the primary ideals of $D$.

Similarly, if $D$ is an $h$-local domain (meaning that $D$ is a locally finite domain such that every nonzero prime ideal is contained in only one maximal ideal) then  $\{D_M\mid M\in\Max(D)\}$ is a Jaffard family, and thus every length function $\ell$ is determined at the local level.
\end{ex}

\section{Primary ideals}\label{sect:primary}
As shown by Proposition \ref{prop:vicev-jaff}, the possibility of decomposing every length function $\ell$ as a sum of $\ell\otimes T$, as $T$ ranges in a (fixed) family $\Theta$, is a quite special property, since it is usually not easy to find Jaffard families of an integral domain $D$. In this section, and in the next one, we shall try to see when a decomposition $\ell=\sum_{T\in\Theta}\ell\otimes T$ can be reached if we allow $\Theta$ to be dependent on $\ell$; to do so, we must treat length functions in a more intrinsic way. 
\begin{lemma}\label{lemma:annZ}
Let $\ell$ be a length function of $D$, and let $M$ be a $D$-module. Then, $\ell(M)=0$ if and only if $\ell(D/\Ann(x))=0$ for all $x\in M$.
\end{lemma}
\begin{proof}
If $\ell(M)=0$, then $\ell(N)=0$ for all submodules $N$ of $M$. In particular, this happens for $N=xD$, for every $x\in M$; however, $xD\simeq D/\Ann(x)$, and thus $\ell(D/\Ann(x))=0$.

Conversely, suppose $\ell(D/\Ann(x))=0$ for all $x\in M$. By upper continuity, it is enough to prove that $\ell(N)=0$ for all finitely generated submodule $N$ of $M$. Let thus $N=\langle x_1,\ldots,x_n\rangle$; then, by \cite[Proposition 2.2]{zanardo_length}, $\ell(N)=\sum_i\ell(N_{i+1}/N_i)$, where $N_k:=\langle x_1,\ldots,x_k\rangle$. However, $N_{i+1}/N_i$ is a cyclic $D$-module, generated by $y:=x_{i+1}+N_i$; in particular, it is isomorphic to $D/\Ann(y)$. However, $\Ann(y)\supseteq\Ann(x_{i+1})$, and thus $\ell(D/\Ann(y))\leq\ell(D/\Ann(x_{i+1}))=0$. Therefore, $\ell(N)=0$, and the claim is proved.
\end{proof}

The following is a slight generalization of \cite[Lemma 2]{northcott_length}, of which we follow the proof.
\begin{lemma}\label{lemma:zerodiv}
Let $R$ be a ring (not necessarily an integral domain); let $\alpha,\beta\in R$. Then, the following hold.
\begin{enumerate}[(a)]
\item\label{lemma:zerodiv:leq} $\ell(R/\alpha\beta R)\leq\ell(R/\alpha R)+\ell(R/\beta R)$.
\item\label{lemma:zerodiv:ug} If $\beta$ is not a zero-divisor in $R$, then $\ell(R/\alpha\beta R)=\ell(R/\alpha R)+\ell(R/\beta R)$.
\item\label{lemma:zerodiv:0} If $\ell(R)<\infty$ and $\beta$ is not a zero-divisor, then $\ell(R/\beta R)=0$.
\end{enumerate}
\end{lemma}
\begin{proof}
Consider the exact sequence
\begin{equation*}
0\longrightarrow \beta R/\alpha\beta R\longrightarrow R/\alpha\beta R\longrightarrow R/\beta R\longrightarrow 0.
\end{equation*}
Then, by additivity,
\begin{equation*}
\ell(R/\alpha\beta R)=\ell(R/\beta R)+\ell(\beta R/\alpha\beta R);
\end{equation*}
furthermore, $\ell(\beta R/\alpha\beta R)\leq\ell(R/\alpha R)$ since multiplication by $\beta$ induces a surjection $R/\alpha R\twoheadrightarrow \beta R/\alpha\beta R$. \ref{lemma:zerodiv:leq} is proved.

If $\beta$ is not a zero-divisor, then multiplication by $\beta$ is an isomorphism, and thus \ref{lemma:zerodiv:ug} holds. In particular, if $\alpha=0$ we have
\begin{equation*}
\ell(R)=\ell(R/\beta R)+\ell(R),
\end{equation*}
and if $\ell(R)$ is finite then $\ell(R/\beta R)$ must be 0. Thus, \ref{lemma:zerodiv:0} holds.
\end{proof}

\begin{lemma}\label{lemma:oltre}
Let $D$ be an integral domain and $\ell$ be a length function on $D$. Let $Q$ be a $P$-primary ideal such that $\ell(D/Q)<\infty$, and let $I$ be an ideal such that $Q\subsetneq I\nsubseteq P$. Then, $\ell(D/I)=0$.
\end{lemma}
\begin{proof}
Let $x\in I\setminus P$: then, $(Q:x)=Q$ since $Q$ is $P$-primary \cite[Lemma 4.4(iii)]{atiyah}, and in particular $\overline{x}$ is not a zero-divisor in $D/Q=:R$. By Lemma \ref{lemma:zerodiv}, it follows that $\ell_R(R/\overline{x}R)=0$; however, by definition, $\ell_R(R/\overline{x}R)=\ell(D/(Q,x))$. Since $I\supseteq(Q,x)$, it follows that $\ell(D/I)\leq\ell(D/(Q,x))=0$, as claimed.
\end{proof}

\begin{prop}\label{prop:primary}
Let $D$ be an integral domain, and let $\ell$ be a length function on $D$. Let $P'\subseteq P$ be two prime ideals, and let $Q$ be a $P'$-primary ideal. Then,
\begin{equation*}
\ell(D/Q)=\ell(D_P/QD_P)=\ell_{D_P}(D_P/QD_P)=(\ell\otimes D_P)(D/Q).
\end{equation*}
\end{prop}
\begin{proof}
Since $Q$ is $P'$-primary and $P'\subseteq P$, we have $Q=QD_{P'}\cap D=QD_P\cap D$, and thus there is an exact sequence
\begin{equation*}
0\longrightarrow D/Q\xlongrightarrow{\iota} D_P/QD_P\xlongrightarrow{\pi} N\longrightarrow 0
\end{equation*}
for some $D$-module $N$. If $\ell(D/Q)=\infty$, then also $\ell(D_P/QD_P)=\infty$ and we are done. Suppose $\ell(D/Q)<\infty$, and let $z=\pi(x)\in N$. Consider $I:=\Ann(z)$: then, $I$ contains $Q$ since $Q\subseteq\Ann(x)$. Furthermore, if $x=x'+QD_P\in D_P/QD_P$, there is an $s\in D\setminus P$ such that $sx'\in D$; hence, $sx\in \iota(D/Q)$, or equivalently $0=\pi(sx)=sz$. Therefore, $I\nsubseteq P$, and thus $I\nsubseteq P'$; by Lemma \ref{lemma:oltre}, we have $\ell(D/I)=0$. By Lemma \ref{lemma:annZ}, it follows that $\ell(N)=0$, and thus $\ell(D/Q)=\ell(D_P/QD_P)$, as claimed.

The second and the third equalities come from the definitions.
\end{proof}

In particular, we can obtain a version of V\'amos' results in the vein of Theorem \ref{teor:jaff-scompo}.
\begin{prop}\label{prop:noeth}
Let $D$ be a Noetherian domain, let $\ell$ be a length function on $D$ and let $\Sigma(\ell):=\{P\in\Spec(D)\mid \ell(D/P)>0\}$. Then,
\begin{equation*}
\ell=\sum_{P\in\Sigma(\ell)}\ell\otimes D_P.
\end{equation*}
\end{prop}
\begin{proof}
Let $\ell^\sharp:=\sum_{P\in\Sigma(\ell)}\ell\otimes D_P$; then, $\ell^\sharp$ is a length function. By \cite[Corollary to Lemma 2]{vamos-additive}, to show that $\ell=\ell^\sharp$ it is enough to show that $\ell(D/Q)=\ell^\sharp(D/Q)$ for every prime ideal $Q$ of $D$.

If $Q\notin\Sigma(\ell)$, then $\ell(D/Q)=0$ (by definition of $\Sigma(\ell)$). Furthermore, any ideal $I$ properly containing $Q$ satisfies $\ell(D/I)=0$ (by \cite[Lemma 3]{vamos-additive} or Lemma \ref{lemma:oltre}), and thus no prime ideal containing $Q$ belongs to $\Sigma(\ell)$. However, if $Q\nsubseteq P$ then
\begin{equation*}
(\ell\otimes D_P)(D/Q)=\ell(D/Q\otimes D_P)=\ell(D_P/QD_P)=\ell(0)=0,
\end{equation*}
and thus also $\ell^\sharp(D/Q)=0$.

Suppose now that $Q\in\Sigma(\ell)$. Then, using the previous reasoning we have
\begin{equation*}
\ell^\sharp(D/Q)=\sum_{\substack{P\in\Sigma(\ell)\\ P\supseteq Q}}(\ell\otimes D_P)(D/Q)=\sum_{\substack{P\in\Sigma(\ell)\\ P\supseteq Q}}\ell(D_P/QD_P)=\sum_{\substack{P\in\Sigma(\ell)\\ P\supseteq Q}}\ell(D/Q),
\end{equation*}
with the last equality coming from Proposition \ref{prop:primary}.

If $Q$ is a maximal element of $\Sigma(\ell)$, then $\ell^\sharp(D/Q)=\ell(D/Q)$, as claimed. Suppose $Q$ is not maximal in $\Sigma(\ell)$; then, $\ell(D/Q)=\infty$, since $\ell(D/Q)>0$ (being $Q\in\Sigma(\ell)$) and $\ell(D/Q)$ cannot be finite (using again \cite[Lemma 3]{vamos-additive}/Lemma \ref{lemma:oltre}). Hence, both $\ell(D/Q)$ and $\ell^\sharp(D/Q)$ are infinite.

Therefore, $\ell(D/Q)=\ell^\sharp(D/Q)$ for every prime ideal $Q$, and thus $\ell=\ell^\sharp$, as claimed.
\end{proof}

More generally, Proposition \ref{prop:primary} shows that, for a $P$-primary ideal $Q$, the value of $\ell(D/Q)$ depends only on $\ell_{D_P}$, which is a length function over $D_P$. We now want to extend this result to ideals having a primary decomposition; we premise a lemma, which already appeared, without proof, in \cite{vamos-additive}.
\begin{lemma}\label{lemma:Grassmann}
Let $I,J$ be ideals of a ring $R$ (not necessarily a domain), and let $\ell$ be a length function on $R$. Then,
\begin{equation*}
\ell(R/I)+\ell(R/J)=\ell(R/(I+J))+\ell(R/(I\cap J)).
\end{equation*}
\end{lemma}
\begin{proof}
To simplify the notation, let $\tau(A):=\ell(R/A)$ for every ideal $A$. From the two exact sequences
\begin{equation*}
0\longrightarrow\frac{I}{I\cap J}\longrightarrow\frac{D}{I\cap J}\longrightarrow\frac{D}{I}\longrightarrow 0
\end{equation*}
and
\begin{equation*}
0\longrightarrow\frac{I+J}{J}\longrightarrow\frac{D}{J}\longrightarrow\frac{D}{I+J}\longrightarrow 0
\end{equation*}
and additivity, we have
\begin{equation*}
\begin{cases}
\tau(I\cap J)=\tau(I)+\ell(I/(I\cap J))\\
\tau(J)=\tau(I+J)+\ell((I+J)/J).
\end{cases}
\end{equation*}
Hence,
\begin{equation*}
\tau(I)+\tau(J)+\ell(I/(I\cap J))=\tau(I\cap J)+\tau(I+J)+\ell((I+J)/J).
\end{equation*}
Since $I/(I\cap J)\simeq (I+J)/J$, the claim follows if $\ell(I/(I\cap J))<\infty$. If not, then $\tau(J)\geq\ell((I+J)/J)=\infty$, and thus $\tau(I\cap J)\geq\tau(J)=\infty$; hence, both sides are infinite, and the claim again follows.
\end{proof}

\begin{prop}\label{prop:primdecomp}
Let $D$ be an integral domain, and let $\ell$ be a length function on $D$. Let $I$ be an ideal of $D$ having a primary decomposition $Q_1\cap\cdots\cap Q_n$. Then,
\begin{equation*}
\ell(D/I)=\sum_{i=1}^n\ell(D/Q_i).
\end{equation*}
\end{prop}
\begin{proof}
Set $\tau(A):=\ell(D/A)$ for every ideal $A$ of $D$.

If $\tau(Q_i)=\infty$ for some $i$, then $\tau(I)=\infty$ and we are done. Suppose $\tau(Q_i)<\infty$ for every $i$: we proceed by induction on the number $n$ of components of $I$. If $n=1$ the claim is obvious. Suppose the claim holds up to $n-1$; without loss of generality, $P_1:=\rad(Q_1)$ is a minimal prime of $I$. Then, $I=Q_1\cap J$, where $J:=Q_2\cap\cdots\cap Q_n$. By Lemma \ref{lemma:Grassmann}, we have
\begin{equation*}
\tau(I)+\tau(Q_1+J)=\tau(Q_1)+\tau(J).
\end{equation*}
The ideal $Q_1+J$ is not contained in $P_1$, for otherwise $J\subseteq P_1$, which is impossible since no primary component of $J$ is contained in $P_1$. By Lemma \ref{lemma:oltre}, this implies that $\tau(Q_1+J)=0$, and thus $\tau(I)=\tau(Q_1)+\tau(J)$; the claim now follows by induction.
\end{proof}

\section{Pr\"ufer domains}\label{sect:prufer}
In this section, we obtain a standard representation for length functions on Pr\"ufer domains where every ideal has only finitely many minimal primes; the groundwork for it is the following extension of Proposition \ref{prop:primary}.
\begin{prop}\label{prop:radprimo}
Let $D$ be a Pr\"ufer domain, and let $\ell$ be a length function on $D$. Let $I$ be an ideal of $D$ such that $P:=\rad(I)$ is prime. Then, $\ell(D/I)=\ell(D_P/ID_P)=(\ell\otimes D_P)(D/I)$.
\end{prop}
\begin{proof}
Let $J:=ID_P\cap D$; then, $J$ is a $P$-primary ideal, and by Proposition \ref{prop:primary} we have $\ell(D/J)=\ell(D_P/JD_P)=\ell(D_P/ID_P)$, with the last equality coming from the fact that $JD_P=ID_P$. Thus, we must prove that $\ell(D/I)=\ell(D/J)$.

If $\ell(D/J)=\infty$ then $\ell(D/I)\geq\ell(D/J)=\infty$ and we are done. Suppose $\ell(D/J)<\infty$, and consider the exact sequence
\begin{equation*}
0\longrightarrow J/I\longrightarrow D/I\longrightarrow D/J\longrightarrow 0.
\end{equation*}
Then, $\ell(D/I)=\ell(D/J)+\ell(J/I)$, and thus we need to show that $\ell(J/I)=0$.

Let $x\in J$, and let $A:=(I:x)$; we claim that $P\subsetneq A$. Indeed, let $M\in\Max(D)$: if $I\nsubseteq M$ (equivalently, if $J\nsubseteq M$) then $AD_M=(ID_M:x)=D_M$ contains $PD_M=D_M$. Suppose $I\subseteq M$: then, $AD_M=(ID_M:xD_M)$. If $AD_M\nsupseteq PD_M$, then $AD_M\subsetneq PD_M$; localizing further at $D_P$, we have $AD_P\subseteq PD_P$. However, $AD_P=(ID_P:x)=D_P$, a contradiction. Hence, $P\subseteq A$; but since $AD_P\neq PD_P$, we must also $P\neq A$ and so $P\subsetneq A$.

Since $\ell(D/J)<\infty$, we can now apply Lemma \ref{lemma:oltre}, obtaining $\ell(D/A)=0$; by Lemma \ref{lemma:annZ}, we have $\ell(J/I)=0$ (since $(I:x)$ is equal to the annihilator of $x+I$ in $J/I$) and thus $\ell(D/I)=\ell(D/J)$, as claimed.
\end{proof}

\begin{defin}
Let $D$ be a Pr\"ufer domain, and let $\ell$ be a length function on $D$. The \emph{total spectrum} of $\ell$ is
\begin{equation*}
\Specell(\ell):=\{P\in\Spec(D)\mid \ell(D/Q)>0\text{~for some~}P\text{-primary ideal~}Q\}.
\end{equation*}
\end{defin}

\begin{lemma}\label{lemma:Spect}
Let $D$ be a Pr\"ufer domain, $\ell$ a length function on $D$, and let $P'\subsetneq P$ be prime ideals such that $P\in\Specell(\ell)$. Then, $\ell(D/P')=\infty$, and in particular $P'\in\Specell(\ell)$.
\end{lemma}
\begin{proof}
Let $L$ be a $P$-primary ideal. Then,
\begin{equation*}
P'=P'D_P\cap D\subseteq LD_P\cap D=L.
\end{equation*}
If $\ell(D/P')<\infty$, then by Lemma \ref{lemma:oltre} (applied with $Q=P'$) we would have $\ell(D/L)=0$. Since $L$ was an arbitrary $P$-primary ideal, it would follow that $P\notin\Specell(\ell)$, against the hypothesis. Hence, $\ell(D/P')=\infty$.
\end{proof}

The total spectrum of $\ell$ is exactly the set we are looking for.
\begin{teor}\label{teor:prufer}
Let $D$ be a Pr\"ufer domain such that every ideal of $D$ has only finitely many minimal primes. For every length function $\ell$ on $D$, we have
\begin{equation*}
\ell=\sum_{P\in\Specell(\ell)}\ell\otimes D_P.
\end{equation*}
\end{teor}
\begin{proof}
Let $\ell^\sharp:=\sum_{P\in\Specell(\ell)}\ell\otimes D_P$; then, $\ell^\sharp$ is a length function by Lemma \ref{lemma:sommalungh}. To show that $\ell=\ell^\sharp$ it is enough to show that $\ell(D/I)=\ell^\sharp(D/I)$ for every ideal $I$ of $D$. Let thus $I$ be an ideal of $D$, and let $\{P_1,\ldots,P_n\}$ be the minimal primes of $I$.

For each $i$, let $T_i:=\bigcap\{D_Q\mid Q\in \V(P_i)\}$, where $\V(A):=\{P\in\Spec(D)\mid A\subseteq P\}$; then, $T_i$ is a Pr\"ufer domain whose prime ideals are the extension of the prime ideals comparable with $P_i$. Let $J_i:=IT_i\cap D$; then, $\rad(J_i)=P_i$, and since every maximal ideal containing $I$ survives in some $T_i$, we have $I=J_1\cap\cdots\cap J_n$. Fix $i$, and let $L_i:=\bigcap_{k\neq i}J_k$: then, the minimal primes of $L_i$ are $P_1,\ldots,P_{i-1},P_{i+1},\ldots,P_n$. In particular, since $\rad(J_i)=P_i$ and $\Spec(D)$ is a tree, there are no prime ideals containing both $J_i$ and $L_i$; thus, $J_i+L_i=D$. By Lemma \ref{lemma:Grassmann}, it follows that
\begin{equation*}
\ell'(D/I)=\ell'(D/(J_i\cap L_i))=\ell'(D/J_i)+\ell'(D/L_i)
\end{equation*}
for every length function $\ell'$; by induction, it follows that $\ell'(D/I)=\sum_i\ell'(D/J_i)$ for every $\ell'$. In particular, it holds for $\ell'=\ell$ and for $\ell'=\ell^\sharp$; hence, we need only to prove that $\ell(D/J_i)=\ell^\sharp(D/J_i)$ for every $J_i$, or equivalently that $\ell(D/J)=\ell^\sharp(D/J)$ for every $J$ such that $\rad(J)=P\in\Spec(D)$.

By Proposition \ref{prop:radprimo}, $\ell(D/J)=(\ell\otimes D_P)(D/J)$. On the other hand,
\begin{equation*}
\ell^\sharp(D/J)=\sum_{Q\in\Specell(\ell)}(\ell\otimes D_Q)(D/J)= \sum_{Q\in\Specell(\ell)}\ell(D_Q/JD_Q).
\end{equation*}
If $Q\nsupseteq P$, then $JD_Q=D_Q$, and so $\ell(D_Q/JD_Q)=0$. Hence,
\begin{equation}\label{eq:sharp-supset}
\ell^\sharp(D/J)=\sum_{\substack{Q\in\Specell(\ell)\\ Q\supseteq P}}\ell(D_Q/JD_Q)=\sum_{\substack{Q\in\Specell(\ell)\\ Q\supseteq P}}(\ell\otimes D_Q)(D/J).
\end{equation}
If $P$ is maximal in $\Sigma(\ell)$, then \eqref{eq:sharp-supset} reduces to $\ell^\sharp(D/J)=(\ell\otimes D_P)(D/J)$, and so is equal to $\ell(D/J)$. If $P$ is not maximal, then by Lemma \ref{lemma:Spect} $\ell(D/P)=\infty$, and so both $\ell(D/J)$ and $\ell^\sharp(D/J)$ are infinite; in particular, they are equal. The claim is proved.
\end{proof}

This theorem does not hold for general Pr\"ufer domain. We shall see an example at the end of the paper (Example \ref{ex:global}).

\medskip

Since $\ell\otimes D_P=(\ell')^{D}$, where $\ell'$ is a length function on $D_P$, the previous theorem effectively reduces the study of the length functions on $D$ to the local case. If $V$ is a valuation domain, the length functions on $V$ has been studied in \cite{northcott_length} and \cite{zanardo_length}: they can be divided into the following four classes.
\begin{itemize}
\item Torsion singular length functions: if $P\in\Spec(V)$, we define $t_P$ as
\begin{equation*}
t_P(M):=\begin{cases}
0 & \text{if~}M\text{~is a torsion~}V/P\text{-module}\\
\infty & \text{otherwise}.
\end{cases}
\end{equation*}

\item Idempotent singular length functions: if $P\in\Spec(V)$ is idempotent, we define $i_P$ as
\begin{equation*}
i_P(M):=\begin{cases}
0 & \text{if~}M\text{~is a~}V/P\text{-module}\\
\infty & \text{otherwise}.
\end{cases}
\end{equation*}

\item $L$-rank functions: if $P\in\Spec(V)$ is idempotent, and $\alpha\in\insR^+$, then $\ell=\alpha\cdot\rk_P$ for some $\alpha\in\insR^+$, where
\begin{equation*}
\rk_P(M):=\begin{cases}
\rank_{V/P}(M) & \text{if~}M\text{~is a~}V/P\text{-module}\\
\infty & \text{otherwise}.
\end{cases}
\end{equation*}

\item Valuative length functions: let $P\in\Spec(D)$ be a branched prime ideal. Let $Q$ be the largest prime ideal contained in $P$, and let $v$ be a valuation on $D_P/Q$. We define $L_v$ as the function
\begin{equation*}
L_v(M):=\sup\sum_{i=1}^s\inf\left\{v(\alpha)\mid \alpha\in\Ann(E_i/E_{i-1})\right\},
\end{equation*}
where the supremum is taken over all finite chains of submodules $(0)=E_0\subsetneq E_1\subsetneq\cdots\subsetneq E_s=M$.
\end{itemize}

\begin{oss}
~\begin{enumerate}
\item The four classes of length functions are pairwise disjoint; however, the classes of idempotent singular length functions and of $L$-rank functions could be merged by considering the functions of type $\alpha\cdot\rk_P$ for $\alpha\inR^{\geq 0}$, i.e., by allowing $\alpha=0$ in the definition of $L$-rank function (and using the convention $0\cdot\infty=\infty$). However, in view of the study of singular length functions (see Corollary \ref{cor:singdisc} and Section \ref{sect:singular} from below Proposition \ref{prop:spectral} onwards) it is more useful to consider them separately.
\item $i_{(0)}$ is the zero length function.
\item The rank function on $V$ is $\rk_{(0)}$; on the other hand, if $M$ is the maximal ideal of $V$, then the ``usual'' length is $\rk_M$ if $M$ is idempotent, while if $M=mV$ is principal (and thus, in particular, branched) it is equal to $L_v$, where $v$ is a rank-one valuation on $V/Q$ normalized in such a way that $v(\overline{m})=1$.
\end{enumerate}
\end{oss}

Suppose now that $D$ is a Pr\"ufer domain, $P\in\Spec(D)$, and let $\ell$ be a length function on the valuation domain $D_P$. Calculating $\ell^D(D/I)$ (where $I$ is an ideal of $D$) corresponds to calculating the values of $\ell(D_P/ID_P)$, which can be done easily by considering the various classes of length function of valuation domains; the results are the following.
\begin{itemize}[itemsep=3ex]
\item $\displaystyle{
(t_{PD_P})^D(D/I)=\begin{cases}
0 & \text{if~}PD_P\subsetneq ID_P\\
\infty & \text{otherwise};
\end{cases}}$
\item $\displaystyle{(i_{PD_P})^D(D/I)=\begin{cases}
0 & \text{if~}PD_P\subseteq ID_P\\
\infty & \text{otherwise};
\end{cases}
}$
\item $\displaystyle{(\rk_{PD_P})^D(D/I)=\begin{cases}
\rank_{\frac{D}{P}}\left(\frac{D_P}{ID_P}\right) & \text{if~}PD_P\subseteq ID_P\\
\infty & \text{otherwise}
\end{cases}=\begin{cases}
0 & \text{if~}PD_P\subsetneq ID_P\\
1 & \text{if~}PD_P=ID_P\\
\infty & \text{otherwise}
\end{cases}
}$
\item $\displaystyle{
(L_v)^D(D/I)=\begin{cases}
0 & \text{if~}I\nsubseteq P\\
\inf v(ID_P/QD_P) & \text{if~}P\text{~is minimal over~}I\\
\infty & \text{otherwise}.
\end{cases}
}$
\end{itemize}

\medskip

Let now $\ell$ be a length function on the Pr\"ufer domain $D$ (not necessarily satisfying the hypothesis of Theorem \ref{teor:prufer}). To each class of length functions on valuation domains we can associate a subset of the spectrum of $D$:
\begin{itemize}[itemsep=1ex]
\item $\Sigma_t(\ell):=\{P\in\Spec(D)\mid \ell_{D_P}=t_{PD_P}\}$;
\item $\Sigma_i(\ell):=\{P\in\Spec(D)\mid \ell_{D_P}=i_{PD_P}\}$;
\item $\Sigma_r(\ell):=\{P\in\Spec(D)\mid \ell_{D_P}=\alpha\rk_{PD_P}$ for some $\alpha\in\insR^+\}$;
\item $\Sigma_v(\ell):=\{P\in\Spec(D)\mid \ell_{D_P}=L_v$ for some rank-one valuation $v$ on $D_P/QD_P\}$.
\end{itemize}

\begin{lemma}\label{lemma:sigma}
Let $D$ be a Pr\"ufer domain, $\ell$ a length function.
\begin{enumerate}[(a)]
\item\label{lemma:sigma:disj} The four sets $\Sigma_t(\ell)$, $\Sigma_i(\ell)$, $\Sigma_r(\ell)$ and $\Sigma_v(\ell)$ are pairwise disjoint.
\item\label{lemma:sigma:union} $\Sigma(\ell)=\Sigma_t(\ell)\cup\Sigma_i(\ell)\cup\Sigma_r(\ell)\cup\Sigma_v(\ell)$.
\item\label{lemma:sigma:downward} If $P,Q\in\Spec(D)$ are such that $P\in\Sigma(\ell)$ and $Q\subsetneq P$, then $Q\in\Sigma_t(\ell)$.
\end{enumerate}
\end{lemma}
\begin{proof}
\ref{lemma:sigma:disj} follows from the fact that the four classes of length functions on valuation domains are disjoint; \ref{lemma:sigma:union} follows from the calculations of $\ell^D(D/I)$. \ref{lemma:sigma:downward} is another version of Lemma \ref{lemma:Spect}, since $\ell(D/Q)=\infty$ if and only if $Q\in\Sigma_t(\ell)$.
\end{proof}

With this terminology, we get a restatement of Theorem \ref{teor:prufer} and a way to characterize singular and discrete length functions.
\begin{cor}\label{cor:prufer-separaz}
Let $D$ be a Pr\"ufer domain such that every ideal of $D$ has only finitely many minimal primes, and let $\ell$ be a length function on $D$. Then,
\begin{equation*}
\ell=\sum_{P\in\Sigma_t(\ell)}(t_{PD_P})^D+ \sum_{P\in\Sigma_i(\ell)}(i_{PD_P})^D+ \sum_{P\in\Sigma_r(\ell)}\ell(D/P)(\rk_{PD_P})^D+ \sum_{P\in\Sigma_v(\ell)}(L_{v_P})^D,
\end{equation*}
where $v_P$ is a rank-one valuation on $D_P/QD_P$ (and $Q$ is the prime ideal directly below $P$).
\end{cor}

\begin{cor}\label{cor:singdisc}
Let $D$ be a Pr\"ufer domain such that every ideal of $D$ has only finitely many minimal primes, and let $\ell$ be a length function on $D$.
\begin{enumerate}[(a)]
\item $\ell$ is singular if and only if $\Sigma_r(\ell)=\Sigma_v(\ell)=\emptyset$.
\item $\ell$ is discrete if and only if $\Sigma_v(\ell)$ does not contain any idempotent prime and the family $\{\ell(D/P)\mid P\in\Sigma_r(\ell)\cup\Sigma_v(\ell)\}$ is discrete.
\end{enumerate}
\end{cor}

The representation of Corollary \ref{cor:prufer-separaz} can also be seen as a way to define a length function: given $\Sigma_1,\Sigma_2,\Sigma_3,\Sigma_4\subseteq\Spec(D)$, $\alpha_P\in\insR^+$ and a valuation on $D_P/LD_P$, we can define a length function $\ell$ by
\begin{equation*}
\ell:=\sum_{P\in\Sigma_1}(t_{PD_P})^D+ \sum_{P\in\Sigma_2}(i_{PD_P})^D+ \sum_{P\in\Sigma_3}\alpha_P(\rk_{PD_P})^D+ \sum_{P\in\Sigma_4}(L_{v_P})^D.
\end{equation*}
In general, there is no guarantee that this representation is the same as the one obtained in the corollary, i.e., the conditions $\Sigma_1=\Sigma_t(\ell)$, $\Sigma_2=\Sigma_i(\ell)$, etc. need not to be satisfied; indeed, being arbitrary, the sets $\Sigma_j$ usually do not satisfy conditions \ref{lemma:sigma:disj} and \ref{lemma:sigma:downward} of Lemma \ref{lemma:sigma}. For example, if $P\nsubseteq Q$ are two prime ideals, the families $\Sigma_1=\{Q\}$, $\Sigma_2=\Sigma_3=\Sigma_4=\emptyset$ and $\Sigma'_1=\{Q,P\}$, $\Sigma'_2=\Sigma'_3=\Sigma'_4=\emptyset$ give rise to the same $\ell$. 

However, we can obtain uniqueness just by excluding the more obvious problems. To express it in a slightly less unwieldy way, we introduce the following definition.
\begin{defin}\label{def:lay}
Let $(\mathcal{P},\leq)$ be a partially ordered set. A family $\{X_1,\ldots,X_n\}$ of subsets of $\mathcal{P}$ is a \emph{layered family with core $X_k$} if:
\begin{itemize}
\item $X_i\cap X_j=\emptyset$ if $i\neq j$;
\item if $x\in\bigcup_iX_i$ and $y<x$, then $y\in X_k$.
\end{itemize}
\end{defin}
Under this terminology, parts \ref{lemma:sigma:disj} and \ref{lemma:sigma:downward} of Lemma \ref{lemma:sigma} can be reparaphrased by saying that $\{\Sigma_t(\ell),\Sigma_i(\ell),\Sigma_r(\ell),\Sigma_v(\ell)\}$ is a layered family with core $\Sigma_t(\ell)$.

\begin{prop}\label{prop:unicita-prufer}
Let $D$ be a Pr\"ufer domain. Let $\{\Sigma_1,\Sigma_2,\Sigma_3,\Sigma_4\}$ and $\{\Sigma'_1,\Sigma'_2,\Sigma'_3,\Sigma'_4\}$ be layered families of $\Spec(D)$ with core $\Sigma_1$ and $\Sigma'_1$, respectively, and suppose that:
\begin{itemize}
\item each prime in $\Sigma_2$, $\Sigma_3$, $\Sigma'_2$ and $\Sigma'_3$ is idempotent;
\item each prime in $\Sigma_4$ and $\Sigma'_4$ is branched;
\item if $P$ is unbranched and every prime properly contained in $P$ belongs to $\Sigma_1$ (respectively, $\Sigma'_1$), then $P\in\Sigma_1\cup\Sigma_2\cup\Sigma_3$ (resp., $P\in\Sigma'_1\cup\Sigma'_2\cup\Sigma'_3$).
\end{itemize}
Furthermore, for every $P\in\Sigma_3$ (resp., $P\in\Sigma'_3$), let $\alpha_P\in\insR^+$ (resp., $\alpha'_P\in\insR^+$), and for every $P\in\Sigma_4$ (resp., $P\in\Sigma'_4$) let $v_P$ (resp., $v'_P$) be a  valuation relative to $D_P/QD_P$, where $Q$ is the largest prime properly contained in $P$. Let $\ell$ and $\ell'$ be the length functions
\begin{equation*}
\ell=\sum_{P\in\Sigma_1}(t_{PD_P})^D+ \sum_{P\in\Sigma_2}(i_{PD_P})^D+ \sum_{P\in\Sigma_3}\alpha_P(\rk_{PD_P})^D+ \sum_{P\in\Sigma_4}(L_{v_P})^D,
\end{equation*}
and
\begin{equation*}
\ell'=\sum_{P\in\Sigma'_1}(t_{PD_P})^D+ \sum_{P\in\Sigma'_2}(i_{PD_P})^D+ \sum_{P\in\Sigma'_3}\alpha'_P(\rk_{PD_P})^D+ \sum_{P\in\Sigma'_4}(L_{v'_P})^D.
\end{equation*}
Then, $\ell=\ell'$ if and only if the following hold:
\begin{itemize}
\item $\Sigma_j=\Sigma'_j$ for $j=1,2,3,4$;
\item $\alpha_P=\alpha'_P$ for all $P\in\Sigma_3$;
\item $v_P=v'_P$ for all $P\in\Sigma_4$.
\end{itemize}
\end{prop}
\begin{proof}
If the three properties hold, then clearly $\ell=\ell'$. Suppose now that $\ell=\ell'$. Let $\Sigma:=\bigcup_j\Sigma_j$ and $\Sigma':=\bigcup_j\Sigma'_j$.

Let $P$ be a prime ideal of $D$. Then, a direct calculation shows that
\begin{equation*}
\ell(D/P)=\begin{cases}
0 & \text{if~}P\in\Sigma_2\cup\Sigma_4\text{~or~}P\notin\Sigma\\
\alpha_P & \text{if~}P\in\Sigma_3\\
\infty & \text{if~}P\in\Sigma_1
\end{cases}
\end{equation*}
and analogously (\emph{mutatis mutandis}) for $\ell'$. In particular, $\ell(D/P)=\infty$ if and only $P\in\Sigma_1$ and $\ell'(D/P)=\infty$ if and only if $P\in\Sigma'_1$; since $\ell=\ell'$ it follows that $\Sigma_1=\Sigma'_1$. Likewise, $\ell(D/P)\notin\{0,\infty\}$ if and only if $P\in\Sigma_3$, and analogously for $\ell'$; hence, $\Sigma_3=\Sigma'_3$ and $\alpha_P=\alpha'_P$ for all $P\in\Sigma_3$.

Let now $\mathcal{B}$ and $\mathcal{U}$ be, respectively, be the set of branched and unbranched prime ideals of $D$.

Let $P$ be a branched prime ideal, and consider a $P$-primary ideal $L$ with $L\subsetneq P$. Another calculation shows that
\begin{equation*}
\ell(D/L)=\begin{cases}
0 & \text{if~}P\notin\Sigma\\
\inf v_P(LD_P/QD_P) & \text{if~}P\in\Sigma_4\\
\infty & \text{if~}P\in\Sigma_1\cup\Sigma_2\cup\Sigma_3
\end{cases}
\end{equation*}
where $Q$ is the biggest prime ideal properly contained in $P$. In particular, $P\in\Sigma$ if and only if $\ell(D/L)>0$ (since $L\subsetneq P$); hence, $\Sigma\cap\mathcal{B}=\Sigma'\cap\mathcal{B}$. Furthermore, $\ell(D/L)=\infty$ if and only if $P\in\Sigma_1\cup\Sigma_2\cup\Sigma_3$, and similarly for $\ell'$; since $\Sigma_1=\Sigma'_1$ and $\Sigma_3=\Sigma'_3$ by the previous reasoning, and the $\Sigma_i$ and $\Sigma'_i$ are disjoint, it follows that $\Sigma_2\cap\mathcal{B}=\Sigma'_2\cap\mathcal{B}$.

Let $P$ be an unbranched prime ideal. If $P\notin\Sigma$, then by hypothesis there is a $Q\subsetneq P$ not contained in $\Sigma$; in particular (eventually passing to the minimal prime of an $x\in P\setminus Q$) we can suppose that $Q$ is branched. By the previous reasoning, $Q\notin\Sigma'$, and thus also $P\notin\Sigma'$. By the same reasoning, if $P\notin\Sigma'$ then $P\notin\Sigma$; hence, $\Sigma\cap\mathcal{U}=\Sigma'\cap\mathcal{U}$. Furthermore, $\Sigma_4\cap\mathcal{U}=\emptyset=\Sigma'_4\cap\mathcal{U}$; hence, $\Sigma_2\cap\mathcal{U}=\Sigma'_2\cap\mathcal{U}$.

Putting together the two cases, we see that $\Sigma=\Sigma'$ and $\Sigma_2=\Sigma'_2$; hence, $\Sigma_4=\Sigma'_4$. Moreover, $\inf v_P(LD_P/QD_P)=\inf v'_P(LD_P/QD_P)$ imply that $v_P=v'_P$. The claim is proved.
\end{proof}

\begin{oss}
If we drop the hypothesis on the unbranched prime ideals, the proposition above does not hold. For example, let $P$ be an unbranched prime ideal of $D$, and let $\Delta$ be the set of prime ideals properly contained in $P$. Then,
\begin{equation*}
\left(\sum_{Q\in\Delta}(t_{QD_Q})^D\right)(D/I)=\begin{cases}
0 & \text{if~}QD_Q\subsetneq ID_Q\text{~for every~}Q\in\Delta\\
\infty & \text{otherwise}.
\end{cases}
\end{equation*}
The first condition holds if and only if $PD_P\subseteq ID_P$; hence, $\sum_{Q\in\Delta}(t_{QD_Q})^D=(i_{PD_P})^D$. In the notation of Proposition \ref{prop:unicita-prufer}, this means the length function induced by the families $\Sigma_1=\Delta$, $\Sigma_2=\Sigma_3=\Sigma_4=\emptyset$ and $\Sigma'_1=\Delta$, $\Sigma'_2=\{P\}$, $\Sigma'_3=\Sigma'_4=\emptyset$ are the same, and thus uniqueness does not hold.
\end{oss}

As a consequence, we can prove that (under the hypotheses of Theorem \ref{teor:prufer}) the set $\lunghezzegen(D)$ depends only on $\Spec(D)$ and the idempotence of the primes of $D$.

\begin{teor}\label{teor:corrisp}
Let $A,B$ be Pr\"ufer domains such that every ideal of $A$ and $B$ has only finitely many minimal primes. Suppose that there is a homeomorphism $\phi:\Spec(A)\longrightarrow\Spec(B)$ such that a prime ideal $P$ is idempotent if and only if $\phi(P)$ is idempotent. Then, there is an order isomorphism $\overline{\phi}:\lunghezzegen(A)\longrightarrow\lunghezzegen(B)$ that respects the classes of $\ell\otimes D_P$, i.e., such that $\phi(\Lambda(\ell))=\Lambda(\overline{\phi}(\ell))$ for each $\Lambda\in\{\Sigma_t,\Sigma_i,\Sigma_r,\Sigma_v\}$.
\end{teor}
\begin{proof}
We first note that, if $P\in\Spec(A)$, then $P$ is branched if and only if $\phi(P)$ is branched: indeed, the homeomorphism $\phi$ induces homeomorphisms $\phi_P:\Spec(A_P)\longrightarrow\Spec(B_{\phi(P)})$, and the maximal ideal $M$ of a valuation ring $V$ is branched if and only if $\Spec(V)\setminus\{M\}$ is compact. Since $\Spec(A_P)\setminus\{PA_P\}$ corresponds to $\Spec(B_{\phi(P)})\setminus\{\phi(P)B_{\phi(P)}\}$, we have that $P$ is branched if and only if $\phi(P)$ is branched.

For every branched prime ideal $P$ of $A$, fix a valuation $v_P$ on $A_P/LA_P$ (viewed as a map to $\insR$), where $L$ is the largest prime properly contained in $P$; then, all valuations on $A_P/LA_P$ are in the form $\lambda v_P$, for some $\lambda\in\insR^+$. Similarly, if $P$ is a branched prime ideal of $B$, fix a valuation $w_P$ on $B_P/LB_P$ (with the same notation for $L$).

Let $\ell\in\lunghezzegen(A)$. By Corollary \ref{cor:prufer-separaz}, and with the notation as above, we can write
\begin{equation*}
\ell=\sum_{P\in\Sigma_t(\ell)}(t_{PA_P})^A+ \sum_{P\in\Sigma_i(\ell)}(i_{PA_P})^A+ \sum_{P\in\Sigma_r(\ell)}\ell(A/P)(\rk_{PA_P})^A+ \sum_{P\in\Sigma_v(\ell)}(L_{\lambda_Pv_P})^A;
\end{equation*}
hence, we define
\begin{align*}
\overline{\phi}(\ell) & :=\sum_{Q\in\phi(\Sigma_t(\ell))}(t_{QB_Q})^B+ \sum_{Q\in\phi(\Sigma_i(\ell))}(i_{QB_Q})^B+\\
& +\sum_{Q\in\phi(\Sigma_r(\ell))}\ell(A/\phi^{-1}(Q))(\rk_{QB_Q})^B+ \sum_{Q\in\phi(\Sigma_v(\ell))}(L_{\lambda_{\phi^{-1}(Q)}w_Q})^B.
\end{align*}
By the previous reasoning, every prime of $\phi(\Sigma_i(\ell))$ and $\phi(\Sigma_r(\ell))$ is idempotent and every prime of $\phi(\Sigma_v(\ell))$ is branched; hence, $\overline{\phi}$ is a well-defined map from $\lunghezzegen(A)$ to $\lunghezzegen(B)$. Furthermore, since $\phi$ is a homeomorphism it is straightforward to see that $\{\phi(\Sigma_t(\ell)),\phi(\Sigma_i(\ell)),\phi(\Sigma_r(\ell)),\phi(\Sigma_v(\ell))\}$ is a layered family with core $\phi(\Sigma_t(\ell))$, and that if every prime contained in the branched prime $Q$ of $B$ is in $\phi(\Sigma_t(\ell))$ then $Q$ is in $\phi(\Sigma_t(\ell))\cup\phi(\Sigma_i(\ell))\cup\phi(\Sigma_r(\ell))$. By Proposition \ref{prop:unicita-prufer}, thus, $\overline{\phi}$ is injective, and it respects the classes of $\ell\otimes D_P$.

With the same reasoning, we can build an injective map $\overline{\phi^{-1}}:\lunghezzegen(B)\longrightarrow\lunghezzegen(A)$, and it is an easy verification that $\overline{\phi^{-1}}$ is the inverse of $\overline{\phi}$. Hence, $\overline{\phi}$ and $\overline{\phi^{-1}}$ are bijections between $\lunghezzegen(A)$ and $\lunghezzegen(B)$.

\medskip

Suppose now $\ell_1\leq\ell_2$ are length functions on $A$: we claim that $\overline{\phi}(\ell_1)\leq\overline{\phi}(\ell_2)$. 

It is enough to verify the inequality at $B/J$, where $J$ is an ideal of $B$; furthermore, with the same reasoning of the proof of Theorem \ref{teor:prufer}, we can reduce this verification to the case where $\rad(J)=:Q$ is a prime ideal, and by Proposition \ref{prop:radprimo} we can further suppose that $J$ is a primary ideal. Let $P:=\phi^{-1}(Q)$.

If $P\notin\Sigma(\ell_1)$, then $Q\notin\phi(\Sigma(\ell_1))=\Sigma(\overline{\phi}(\ell_1))$, and thus, by Proposition \ref{prop:radprimo}, $\ell_1(B/J)=0$. Hence, $\overline{\phi}(\ell_1)(B/J)\leq\overline{\phi}(\ell_2)(B/J)$. Suppose thus that $P\in\Sigma(\ell_1)$, i.e., that $Q\in\Sigma(\overline{\phi}(\ell_1))$.

Then, $\overline{\phi}(\ell_1)(B/J)=(\overline{\phi}(\ell_1)\otimes B_Q)(B/J)$. Furthermore, if $Q$ is a prime ideal such that $Q\notin\phi(\Sigma_v(\ell))$ (equivalently, if $P\notin\Sigma_v(\ell)$) then $\ell(A/J)=\overline{\phi}(\ell)(B/Q)$ for every length function $\ell$; in particular, 
\begin{equation*}
\overline{\phi}(\ell_1)(B/Q)=\ell_1(A/P)\leq\ell_2(A/P)=\overline{\phi}(\ell_2)(B/Q).
\end{equation*}
Thus, if $J=Q$ we are done; suppose now that $J\subsetneq Q$.

If $P\notin\Sigma_v(\ell)$ and $L_1,L_2$ are $P$-primary ideals such that $L_1,L_2\subsetneq P$, then $\ell(A/L_1)=\ell(A/L_2)$ (equal to $0$ if $P\notin\Sigma(\ell)$ and to $\infty$ if $P\in\Sigma(\ell)$); hence, if $P\in\Sigma(\ell_1)\setminus\Sigma_v(\ell)$ then $\ell_1(A/L)=\infty$ for every $P$-primary ideal $L\subsetneq P$, which means that $\ell_2(A/L)=\infty$; it follows that also $\overline{\phi}(\ell_1)(B/J)=\infty=\overline{\phi}(\ell_2)(B/J)$, and in particular $\overline{\phi}(\ell_1)(B/J)\leq\overline{\phi}(\ell_2)(B/J)$.

Hence, we only need to consider that case $P\in\Sigma_v(\ell_1)$; in particular, $P$ is branched, and so $Q$ is as well. If $P\notin\Sigma_v(\ell_2)$, then $\ell_2(A/L)=\infty$ for all $P$-primary ideals $L\subsetneq P$; hence $\overline{\phi}(\ell_1)(B/J)\leq\overline{\phi}(\ell_2)(B/J)$ since the latter is equal to $\infty$.

Suppose thus that $P\in\Sigma_v(\ell_1)\cap\Sigma_v(\ell_2)$; let $P'$ be the largest prime ideal of $A$ properly contained in $P$, and let $Q'$ be the largest prime ideal of $B$ properly contained in $Q$. Let $\pi_Q:B_Q\longrightarrow B_Q/Q'B_Q$ be the canonical quotient map. If $QB_Q$ is principal, then $JB_Q=(QB_Q)^n$ for some $n$, and we define $I:=(PA_P)^n\cap A$. If $QB_Q$ is not principal, let $\delta:=\inf\{w_Q(\pi_Q(x))\mid x\in J\}\inR$. If $\pi_P:A_P\longrightarrow A_P/P'A_P$ is the quotient, then we define $I:=\{y\in A\mid v_P(\pi_P(y))\geq\delta\}$. In both cases, $I$ is an ideal of $A$ whose radical is $P$, and by construction $\ell(A/I)=\delta=\overline{\phi}(\ell)(B/J)$ for every length function $\ell$ such that $P\in\Sigma_v(\ell)$; as above, it follows that $\overline{\phi}(\ell_1)(B/J)\leq\overline{\phi}(\ell_2)(B/J)$.

Therefore, $\overline{\phi}(\ell_1)\leq\overline{\phi}(\ell_2)$, and so $\overline{\phi}$ is order-preserving. By symmetry, the same happens for $\overline{\phi^{-1}}=\overline{\phi}^{-1}$; hence, $\overline{\phi}$ is an order isomorphism, as claimed.
\end{proof}

\section{Singular length functions}\label{sect:singular}
In this section, we characterize singular length functions through purely ideal-theoretic means, by using the concept of localizing system (see \cite{localizing-semistar} or \cite[Section 5.1]{fontana_libro}). We denote by $\insid(D)$ the set of ideals of $D$.
\begin{defin}\label{def:locsist}
Let $D$ be an integral domain. A \emph{localizing system} on $D$ is a set $\mathcal{F}\subseteq\insid(D)$ such that:
\begin{itemize}
\item if $I\in\mathcal{F}$ and $I\subseteq J$, then $J\in\mathcal{F}$;
\item if $I\in\mathcal{F}$ and $(J:iD)\in\mathcal{F}$ for all $i\in I$, then $J\in\mathcal{F}$.
\end{itemize}
We denote by $\locsist(D)$ the set of localizing systems on $D$.
\end{defin}

Our next aim is to prove that to every length function can be associated a localizing system, and conversely.
\begin{defin}
Let $\ell$ be a length function. The \emph{zero locus} of $\ell$ is
\begin{equation*}
Z(\ell):=\{I\in\insid(D)\mid \ell(D/I)=0\}.
\end{equation*}
\end{defin}

\begin{prop}
The zero locus of a length function $\ell$ is a localizing system.
\end{prop}
\begin{proof}
If $I\in Z(\ell)$ and $I\subseteq J$ then $\ell(D/J)\leq\ell(D/I)=0$, and so $J\in Z(\ell)$. Suppose $I\in Z(\ell)$ and let $J$ be an ideal such that $(J:iD)\in Z(\ell)$ for every $i\in I$. From the exact sequence
\begin{equation*}
0\longrightarrow (I+J)/J\longrightarrow D/J\longrightarrow D/(I+J)\longrightarrow 0,
\end{equation*}
and from the fact that $I+J\in Z(\ell)$ (since $I+J\supseteq I\in Z(\ell)$), we have $\ell(D/J)=\ell((I+J)/J)$. For every $i+J\in(I+J)/J$, we have $\Ann(i+J)=(J:iD)\in Z(\ell)$ by hypothesis; by Lemma \ref{lemma:annZ}, it follows that $\ell((I+J)/J)=0$, and thus $\ell(D/J)=0$, i.e., $J\in Z(\ell)$. Thus, $Z(\ell)$ is a localizing system.
\end{proof}

Conversely, let $\mathcal{F}$ be a localizing system on $D$. The \emph{length function associated to $\mathcal{F}$} is
\begin{equation*}
\ell_\mathcal{F}(M):=\begin{cases}
0 & \text{if~}\Ann(x)\in \mathcal{F}\text{~for all~}x\in M\\
\infty & \text{otherwise}.
\end{cases}
\end{equation*}

\begin{prop}
For any localizing system $\mathcal{F}$ on $D$, $\ell_\mathcal{F}$ is a length function.
\end{prop}
\begin{proof}
Clearly, $\ell_\mathcal{F}$ is upper continuous. Let
\begin{equation*}
0\longrightarrow M_1\longrightarrow M_2\xlongrightarrow{\pi} M_3\longrightarrow 0
\end{equation*}
be an exact sequence of $D$-modules.

If $\ell_\mathcal{F}(M_2)=0$, then $\Ann(x)\in\mathcal{F}$ for all $x\in M_2$; in particular, $\Ann(y)\in\mathcal{F}$ for all $y\in M_1$ (and so $\ell_\mathcal{F}(M_1)=0$) and $\Ann(z)=\Ann(\pi(x))\supseteq\Ann(x)$ for all $z=\pi(x)\in M_3$ (and so $\ell_\mathcal{F}(M_3)=0$). In particular, $\ell_\mathcal{F}(M_2)=\ell_\mathcal{F}(M_1)+\ell_\mathcal{F}(M_3)$.

Suppose now $\ell_\mathcal{F}(M_2)=\infty$; then, there is an $x\in M_2$ such that $\Ann(x)\notin\mathcal{F}$. Suppose $\ell_\mathcal{F}(M_1)=\ell_\mathcal{F}(M_3)=0$, let $z:=\pi(x)$, and consider the exact sequence
\begin{equation*}
0\longrightarrow xD\cap M_1\longrightarrow xD\longrightarrow zD\longrightarrow 0.
\end{equation*}
Since we supposed $\ell_\mathcal{F}(M_3)=0$, we must have $I:=\Ann(z)\in \mathcal{F}$; furthermore, for every $i\in I$, we have $ix\in xD\cap M_1$, and thus $\Ann(ix)\in \mathcal{F}$. However, $\Ann(ix)=(\Ann(x):iD)$; since $I\in\mathcal{F}$, this would mean that $\Ann(x)\in \mathcal{F}$, against the hypothesis on $x$. Therefore, one between $\ell_\mathcal{F}(M_1)$ and $\ell_\mathcal{F}(M_3)$ is infinite, and thus $\ell_\mathcal{F}(M_2)=\ell_\mathcal{F}(M_1)+\ell_\mathcal{F}(M_3)$. Hence, $\ell_\mathcal{F}$ is a length function.
\end{proof}

\begin{teor}\label{teor:singular}
Let $D$ be an integral domain. The two maps
\begin{equation*}
\begin{aligned}
\singular(D) & \longrightarrow\locsist(D) \\
\ell & \longmapsto Z(\ell)
\end{aligned}\quad\text{and}\quad
\begin{aligned}
\locsist(D) & \longrightarrow\singular(D) \\
\mathcal{F} & \longmapsto \ell_\mathcal{F}
\end{aligned}
\end{equation*}
are bijections, one inverse of the other. Furthermore, if $\locsist(D)$ is endowed with the containment order, they are order-reversing isomorphisms.
\end{teor}
\begin{proof}
Since a singular length function is characterized by its zero locus, we need to show that $Z(\ell_\mathcal{F})=\mathcal{F}$ and that $\ell_{Z(\ell)}=\ell$. Indeed,
\begin{equation*}
Z(\ell_\mathcal{F})=\{I\in\insid(D)\mid \ell_\mathcal{F}(D/I)=0\}.
\end{equation*}
However, for every $x\in D/I$ we have $\Ann(x)\supseteq\Ann(1+I)=I$, and thus $\ell_\mathcal{F}(D/I)=0$ if and only if $I\in Z(\ell)$. Therefore, $Z(\ell_\mathcal{F})=\mathcal{F}$.

On the other hand,
\begin{equation*}
\ell_{Z(\ell)}(M)=\begin{cases}
0 & \text{if~}\Ann(x)\in Z(\ell)\text{~for all~}x\in M\\
\infty & \text{otherwise}.
\end{cases}
\end{equation*}
If $\Ann(x)\in Z(\ell)$ for all $x\in M$, then $\ell(M)=0$ by Lemma \ref{lemma:annZ}, while if $\Ann(x)\notin Z(\ell)$ for some $x\in M$ then $xD\simeq D/\Ann(x)$ and thus, since $\ell$ is singular, $\ell(M)\geq\ell(xD)=\infty$. Thus, $\ell_{Z(\ell)}=\ell$, as claimed.

The last claims follows from the fact that, for singular length functions, $\ell_1\leq\ell_2$ if and only if $Z(\ell_1)\supseteq Z(\ell_2)$.
\end{proof}

Localizing systems are also closely related to the concept of stable semistar operations. Let $\inssubmod(D)$ be the set of $D$-submodules of the quotient field $K$; a \emph{stable semistar operation} on $D$ is a map $\star:\inssubmod(D)\longrightarrow\inssubmod(D)$ such that, for every $I,J\in\inssubmod(D)$ and every $x\in K$:
\begin{itemize}
\item $I\subseteq I^\star$;
\item if $I\subseteq J$, then $I^\star\subseteq J^\star$;
\item $(I^\star)^\star=I^\star$;
\item $(xI)^\star=x\cdot I^\star$;
\item $(I\cap J)^\star=I^\star\cap J^\star$.
\end{itemize}
(A map that satisfies the first four properties is called a \emph{semistar operation}.) We denote by $\inssemistab(D)$ the set of stable semistar operations.

There is a natural bijection between stable semistar operations and localizing system: if $\star$ is a stable semistar operation, then the set $\mathcal{F}^\star:=\{I\in\insid(D)\mid 1\in I^\star\}$ is a localizing system, while if $\mathcal{F}$ is a localizing system then
\begin{equation*}
\star_\mathcal{F}:I\mapsto\bigcup\{(I:E)\mid E\in\mathcal{F}\}
\end{equation*}
is a stable semistar operation; these two correspondences are inverse one of each other \cite[Theorem 2.10]{localizing-semistar}. By composing them with the maps considered in Theorem \ref{teor:singular}, we obtain two bijections
\begin{equation*}
\begin{aligned}
\Phi\colon\inssemistab(D) & \longrightarrow\singular(D)\\
\star & \longmapsto \ell_{\mathcal{F}^\star}
\end{aligned}\quad\text{and}\quad\begin{aligned}
\Phi^{-1}\colon\singular & \longrightarrow\inssemistab(D)\\
\ell & \longmapsto \star_{Z(\ell)}.
\end{aligned}
\end{equation*}
which are order-reversing isomorphisms if $\inssemistab(D)$ is endowed with the order such that $\star_1\leq\star_2$ if $I^{\star_1}\subseteq I^{\star_2}$ for every $I\in\inssubmod(D)$. Note that, in this order, the infimum of a family $\Delta$ is the map sending $I$ to $\bigcap_{\star\in\Delta}I^\star$.

\begin{prop}\label{prop:infsomma}
Let $\Lambda$ be a nonempty set of stable semistar operations on the integral domain $D$. Then,
\begin{equation*}
\Phi(\inf\Lambda)=\sum_{\star\in\Lambda}\Phi(\star).
\end{equation*}
\end{prop}
\begin{proof}
Since $\Phi$ is an order-reversing isomorphism, $\Phi(\inf\Lambda)=\sup\Phi(\Lambda)$; the claim now follows from Lemma \ref{lemma:sommalungh}.
\end{proof}

A special subset of stable semistar operations are \emph{spectral semistar operations}, i.e., closures in the form
\begin{equation*}
s_\Delta:I\mapsto\bigcap_{P\in\Delta}ID_P,
\end{equation*}
where $\Delta\subseteq\Spec(D)$; furthermore, we can suppose that $\Delta$ is \emph{closed by generizations}, i.e., it is such that if $P\in\Delta$ and $Q\subseteq P$ then also $Q\in\Delta$. The corresponding localizing system is 
\begin{equation*}
\mathcal{F}_\Delta:=\{I\in\insid(D)\mid I\nsubseteq P\text{~for every~}P\in\Delta\},
\end{equation*}
while the associated length function $\ell_\Delta$ is such that
\begin{equation*}
\ell_\Delta(D/I)=\begin{cases}
0 & \text{if~}I\nsubseteq P\text{~for every~}P\in\Delta\\
\infty & \text{if~}I\subseteq P\text{~for some~}P\in\Delta,
\end{cases}
\end{equation*}
or, more generally,
\begin{equation*}
\ell_\Delta(M)=\begin{cases}
0 & \text{if~}\Ann(x)\nsubseteq P\text{~for every~}P\in\Delta\text{~and~}x\in M\\
\infty & \text{if~}\Ann(x)\subseteq P\text{~for some~}P\in\Delta\text{~and~}x\in M
\end{cases}
\end{equation*}
for every $D$-module $M$. In particular, if $\Delta=\Spec(D)$ then $\ell_\Delta(M)=0$ if and only if $M=0$, while if $\Delta=\{(0)\}$ then $\ell_\Delta(M)=0$ if and only if $M$ is torsion.

Such length functions have a decomposition like the ones found in Theorems \ref{teor:jaff-scompo} and \ref{teor:prufer}.
\begin{prop}\label{prop:spectral}
Let $D$ be an integral domain, and let $\Delta\subseteq\Spec(D)$ be closed by generizations; let $\ell:=\Phi(s_\Delta)$, and let $\Sigma(\ell):=\{P\in\Spec(D)\mid P\notin Z(\ell)\}$. Then:
\begin{enumerate}[(a)]
\item $\Delta=\Sigma(\ell)$;
\item $\displaystyle{\ell=\sum_{P\in\Sigma(\ell)}\ell\otimes D_P}$.
\end{enumerate}
\end{prop}
\begin{proof}
From the bijections between $\inssemistab(D)$, $\locsist(D)$ and $\singular(D)$, we see that prime $P$ is in $Z(\ell)$ if and only if $1\in P^{s_\Delta}$; since $\Delta$ is closed by generizations, it follows that $P\in Z(\ell)$ if and only if $P\notin\Delta$, and thus $\Delta=\Sigma(\ell)$.

The semistar operation $s_\Delta$ is the infimum of the family $s_{\{P\}}$, as $P$ ranges in $\Delta$; by Proposition \ref{prop:infsomma}, it follows that
\begin{equation*}
\ell=\Phi(s_\Delta)=\sum_{P\in\Delta}\Phi(s_{\{P\}}),
\end{equation*}
and thus we only need to show that $\Phi(s_{\{P\}})=\ell\otimes D_P$; to do so, it is enough to show that their zero locus are equal. We have
\begin{equation*}
Z(\Phi(s_{\{P\}}))=\mathcal{F}^{s_{\{P\}}}=\{I\in\insid(D)\mid I\nsubseteq P\}.
\end{equation*}
If $I\nsubseteq P$, then $D_P/ID_P=0$, and thus
\begin{equation*}
(\ell\otimes D_P)(D/I)=\ell(D/I\otimes_DD_P)=\ell(D_P/ID_P)=\ell(0)=0,
\end{equation*}
i.e., $I\in Z(\ell\otimes D_P)$; on the other hand, if $I\subseteq P$ then, using Proposition \ref{prop:primary},
\begin{equation*}
(\ell\otimes D_P)(D/I)\geq(\ell\otimes D_P)(D/P)=\ell(D/P)=\infty
\end{equation*}
since $P^{s_{\{P\}}}=PD_P$ does not contain $1$. Hence, $I\notin Z(\ell\otimes D_P)$, and thus $Z(\Phi(s_{\{P\}}))=Z(\ell\otimes D_P)$. The claim is proved.
\end{proof}

Let now $D$ be a Pr\"ufer domain, and let $\star\in\inssemistab(D)$. The \emph{normalized stable version} of $\star$ is \cite[Section 4]{stable_prufer}
\begin{equation*}
\widehat{\star}:I\mapsto\bigcap_{P\in\Sigma_1(\star)}ID_P\cap\bigcap_{P\in\Sigma_2(\star)}(ID_P)^{v_{D_P}},
\end{equation*}
where $v_{D_P}$ is the $v$-operation on $D_P$ (i.e., if $J$ is an ideal of $D_P$ then $J^{v_{D_P}}=\bigcap\{yD_P\mid J\subseteq yD_P\}$), and
\begin{equation*}
\begin{aligned}
\Sigma_1(\star):= & \{P\in\Spec(D)\mid 1\notin P^\star\},\\
\Sigma_2(\star):= & \{P\in\Spec(D)\mid 1\in P^\star,~1\notin Q^\star\text{~for some~}P\text{-primary ideal~}Q\}.
\end{aligned}
\end{equation*}
(In the terminology of \cite{stable_prufer}, $\Sigma_1(\star)=:\qspec{\star}(D)$ is the \emph{quasi-spectrum} of $\star$, while $\Sigma_2(\star)=:\psspec{\star}(D)$ is the \emph{pseudo-spectrum}.) By \cite[Proposition 3.4]{stable_prufer}, and in the terminology introduced in Definition \ref{def:lay}, furthermore, $\{\Sigma_1(\star),\Sigma_2(\star)\}$ is a layered family with core $\Sigma_1(\star)$.

This construction is analogous to the passage from a length function $\ell$ to
\begin{equation*}
\ell^\sharp:=\sum_{P\in\Sigma_t(\ell)}(t_{PD_P})^D+\sum_{P\in\Sigma_i(\ell)}(i_{PD_P})^D= \sum_{P\in\Sigma(\ell)}\ell\otimes D_P,
\end{equation*}
as the next proposition shows.
\begin{prop}\label{prop:widehatsharp}
Let $D$ be a Pr\"ufer domain. Then, for every stable star operation $\star$, we have $\Phi(\widehat{\star})=\Phi(\star)^\sharp$.
\end{prop}
\begin{proof}
Let $\Delta$ be the set formed by the functions $d_P:I\mapsto ID_P$, for $P\in\Sigma_1(\star)$, and $v_P:=I\mapsto(ID_P)^{v_{D_P}}$, as $P\in\Sigma_2(\star)$. By Proposition \ref{prop:infsomma}, it follows that
\begin{equation*}
\Phi(\widehat{\star})=\sum_{P\in\Sigma_1(\star)}\Phi(d_P)+\sum_{P\in\Sigma_2(\star)}\Phi(v_P).
\end{equation*}
On the other hand, by unpacking the definitions, we have $\Sigma_t(\Phi(\star))=\Sigma_1(\star)$ and $\Sigma_i(\Phi(\star))=\Sigma_2(\star)$; hence, it is enough to show that $\Phi(d_P)=(t_{PD_P})^D$ and $\Phi(v_P)=(i_{PD_P})^D$ for every $P\in\Spec(D)$, and to do so it is enough to consider their zero loci.

By a direct calculation,
\begin{equation*}
\begin{aligned}
Z(\Phi(d_P))= & \{I\in\insid(D)\mid I\nsubseteq P\}=\\
= & \{I\in\insid(D)\mid PD_P\subsetneq ID_P\}=Z((t_{PD_P})^D);
\end{aligned}
\end{equation*}
analogously,
\begin{equation*}
Z(\Phi(v_P))=\{I\in\insid(D)\mid PD_P\subseteq ID_P\}=Z((i_{PD_P})^D);
\end{equation*}
since $(ID_P)^{v_{D_P}}=D_P$ if and only if $ID_P$ is equal to $D_P$ or to $PD_P$. The claim is proved.
\end{proof}

Suppose now that every ideal of $D$ has only finitely many minimal primes. Then, Theorem \ref{teor:prufer} (and Corollary \ref{cor:singdisc}) can be seen as a version of \cite[Theorem 4.5 and Corollary 4.6]{stable_prufer}: like any singular length function can be written as $\sum_{P\in\Sigma_1}\ell\otimes D_P+\sum_{P\in\Sigma_2}\ell\otimes D_P$, a stable semistar operation can be written as the infimum of the semistar operations $d_P$ (as $P$ ranges in some $\Sigma_1$) and $v_P$ (as $P$ ranges in $\Sigma_2$).

More generally, suppose that $\star$ is a semistar operation which is equal to its normalized stable version, i.e., suppose that there are $\Sigma_1,\Sigma_2\subseteq\Spec(D)$ such that $\{\Sigma_1,\Sigma_2\}$ is a layered family with core $\Sigma_1$ and
\begin{equation*}
\star:I\mapsto\bigcap_{P\in\Sigma_1}ID_P\cap\bigcap_{P\in\Sigma_2}(ID_P)^{v_{D_P}}.
\end{equation*}
Then, by Proposition \ref{prop:widehatsharp}, we see that the corresponding length function $\ell=\ell^\sharp$ can be decomposed as $\ell=\sum_{P\in\Sigma(\ell)}\ell\otimes D_P$. Thus, the fact that a stable semistar operation is determined at the local level (through the closures $d_P$ and $v_P$) corresponds to the fact that the corresponding length function depends exclusively on length functions on the localizations of $D$.

We end the paper with two examples of Pr\"ufer domains of dimension 1 that are not locally finite and whose behavior with respect to decomposition is very different: more precisely, in Example \ref{ex:ad} we present an example where every singular length function can be decomposed (despite the domain not satisfying the hypothesis of Theorem \ref{teor:prufer}), while in Example \ref{ex:global} we give a singular function that can't be decomposed.
\begin{ex}\label{ex:ad}
Let $D$ be an almost Dedekind domain (i.e., an integral domain such that $D_M$ is a discrete valuation ring for every $M\in\Max(D)$), and suppose that there is only a finite (nonzero) number of maximal ideals of $D$ that are not finitely generated. (See \cite{loper_sequence} for explicit examples of domains with this property.) In particular, $D$ is one-dimensional and $\Spec(D)$ is not Noetherian, and so there are ideals with infinitely many minimal primes.

We claim that every singular length function $\ell$ can be written as $\ell=\sum_{M\in\Max(D)}\ell\otimes D_M$, and, to do so, we want to show that every stable semistar operation $\star$ is equal to $\star=s_\Delta$ for some $\Delta\subseteq\Spec(D)$. If not, then by \cite[Theorem 4.12(3)]{localizing-semistar} there is a proper ideal $I$ of $D$ such that $I=I^\star\cap D$ but $P\neq P^\star\cap D$ for every prime ideal $P$ containing $I$; since $D$ is one-dimensional, it follows that $1\in P^\star$ for every $P$ containing $I$, or equivalently that $P^\star=D^\star$.

Suppose that $P=pD$ contains $I$ and is principal; then,
\begin{equation*}
1\in P^\star=(pD)^\star=pD^\star,
\end{equation*}
and thus $1/p\in T:=D^\star$. Hence,
\begin{equation*}
I^\star=(ID)^\star=(ID^\star)^\star\supseteq ID\left[\inv{p}\right].
\end{equation*}
The ideal $ID\left[\inv{p}\right]\cap D$ is not contained in $pD$, since
\begin{equation*}
\left(ID\left[\inv{p}\right]\cap D\right)D_P=ID\left[\inv{p}\right]D_P\cap DD_P=D_P;
\end{equation*}
thus, $ID\left[\inv{p}\right]\cap D\neq I$, a contradiction. Hence, $I$ is not contained in any principal prime ideal; however, this means that $I$ has a primary decomposition, namely $I=\bigcap_i(ID_{M_i}\cap D)$, where $\{M_1,\ldots,M_n\}$ are the maximal ideal of $D$ containing $I$. By Proposition \ref{prop:primdecomp}, it follows that $\ell(D/I)=\sum_i\ell(D/(ID_{M_i}\cap D))$; however, by Proposition \ref{prop:primary},
\begin{equation*}
\ell(D/(ID_M\cap D))=\ell(D_M/ID_M)=\ell(D_M/(MD_M)^k)=0
\end{equation*}
since $D_M$ is a DVR (and so $ID_M=(MD_M)^k$ for some $k$) and $\ell(D_M/MD_M)=\ell(D/M)=0$. It follows that $\ell(D/I)=0$, and thus that $I\in Z(\ell)=\{I\in\insid(D)\mid 1\in I^\star\}$, against the fact that $I=I^\star\cap D$. Hence, $\star$ is spectral and so $\star=s_\Delta$; by Proposition \ref{prop:spectral}, we have 
\begin{equation*}
\ell=\sum_{P\in\Sigma(\ell)}\ell\otimes D_P=\sum_{M\in\Max(D)}\ell\otimes D_M,
\end{equation*}
with the last equality coming from the fact that, if $M\notin\Sigma(\ell)$, then $\ell\otimes D_M$ sends every proper quotient $D/I$ to $0$.
\end{ex}

\begin{ex}\label{ex:global}
Let $D:=\ins{A}$ be the ring of all algebraic integers. By \cite[Example 4.5]{spettrali-eab} and \cite[Example 4.2]{stable_prufer}, we can build a stable semistar operation $\star$ such that $Q^\star=D$ for every primary ideal $Q$, while $D^\star=D$ (and so $(xD)^\star=xD$ for every $x\in D$). The corresponding localizing system contains every ideal contained in only finitely many maximal ideals, but it does not contain any proper principal ideal; hence, the associated length function $\ell$ is such that $\ell(D/I)=0$ if $I$ is contained in only finitely many maximal ideals, while $\ell(D/xD)=\infty$ for all nonunits $x\in D$. 

By Proposition \ref{prop:primary}, if $Q$ is $P$-primary, then
\begin{equation*}
\ell_{D_P}(D_P/QD_P)=\ell(D/Q)=0,
\end{equation*}
while $\ell_{D_P}(D_P)=\infty$; hence, $\ell_{D_P}=t_{(0)}$ for every $P\in\Max(D)$. It follows that $\Sigma_t(\ell)=\{(0)\}$ while $\Sigma_i(\ell)=\Sigma_r(\ell)=\Sigma_v(\ell)=\emptyset$; therefore, setting $\ell^\sharp:=\sum_{P\in\Sigma(\ell)}\ell\otimes D_P$, we have
\begin{equation*}
\ell^\sharp(M)=(t_{(0)})^D(M)=\begin{cases}
0 & \text{if~}M\text{~is a torsion~}D\text{-module}\\
\infty & \text{otherwise},
\end{cases}
\end{equation*}
and thus $\ell^\sharp(D/I)=0$ for every proper ideal $I$ of $D$. In particular, $\ell\neq\ell^\sharp$.

Furthermore, $(\ell\otimes D_M)(D/I)=0$ for every nonzero ideal $I$ and every maximal ideal $M$; hence, we also have
\begin{equation*}
\ell\neq\sum_{P\in\Delta}\ell\otimes D_P.
\end{equation*}
for every family $\Delta\subseteq\Spec(D)$.
\end{ex}

\section*{Acknowledgments}
I would like to thank Luigi Salce for introducing me to length functions and for his feedback on an earlier version of the manuscript.


\begin{thebibliography}{10}
	
	\bibitem{atiyah}
	M.~F. Atiyah and I.~G. Macdonald.
	\newblock {\em Introduction to {C}ommutative {A}lgebra}.
	\newblock Addison-Wesley Publishing Co., Reading, Mass.-London-Don Mills, Ont.,
	1969.
	
	\bibitem{entropy-category}
	Dikran Dikranjan and Anna Giordano~Bruno.
	\newblock Entropy in a category.
	\newblock {\em Appl. Categ. Structures}, 21(1):67--101, 2013.
	
	\bibitem{spettrali-eab}
	Carmelo~A. Finocchiaro, Marco Fontana, and Dario Spirito.
	\newblock Spectral spaces of semistar operations.
	\newblock {\em J. Pure Appl. Algebra}, 220(8):2897--2913, 2016.
	
	\bibitem{fontana_factoring}
	Marco Fontana, Evan Houston, and Thomas Lucas.
	\newblock {\em Factoring {I}deals in {I}ntegral {D}omains}, volume~14 of {\em
		Lecture Notes of the Unione Matematica Italiana}.
	\newblock Springer, Heidelberg; UMI, Bologna, 2013.
	
	\bibitem{localizing-semistar}
	Marco Fontana and James~A. Huckaba.
	\newblock Localizing systems and semistar operations.
	\newblock In {\em Non-{N}oetherian commutative ring theory}, volume 520 of {\em
		Math. Appl.}, pages 169--197. Kluwer Acad. Publ., Dordrecht, 2000.
	
	\bibitem{fontana_libro}
	Marco Fontana, James~A. Huckaba, and Ira~J. Papick.
	\newblock {\em Pr\"ufer {D}omains}, volume 203 of {\em Monographs and Textbooks
		in Pure and Applied Mathematics}.
	\newblock Marcel Dekker Inc., New York, 1997.
	
	\bibitem{gilmer}
	Robert Gilmer.
	\newblock {\em Multiplicative {I}deal {T}heory}.
	\newblock Marcel Dekker Inc., New York, 1972.
	\newblock Pure and Applied Mathematics, No. 12.
	
	\bibitem{knebush-zhang}
	Manfred Knebusch and Digen Zhang.
	\newblock {\em Manis {V}aluations and {P}r\"ufer {E}xtensions. {I}}, volume
	1791 of {\em Lecture Notes in Mathematics}.
	\newblock Springer-Verlag, Berlin, 2002.
	\newblock A new chapter in commutative algebra.
	
	\bibitem{lazard_flat}
	Daniel Lazard.
	\newblock Autour de la platitude.
	\newblock {\em Bull. Soc. Math. France}, 97:81--128, 1969.
	
	\bibitem{loper_sequence}
	Alan Loper.
	\newblock Sequence domains and integer-valued polynomials.
	\newblock {\em J. Pure Appl. Algebra}, 119(2):185--210, 1997.
	
	\bibitem{matlis-tfm}
	Eben Matlis.
	\newblock {\em Torsion-free modules}.
	\newblock The University of Chicago Press, Chicago-London, 1972.
	\newblock Chicago Lectures in Mathematics.
	
	\bibitem{northcott_length}
	Douglas~Geoffrey Northcott and Manfred Reufel.
	\newblock A generalization of the concept of length.
	\newblock {\em Quart. J. Math. Oxford Ser. (2)}, 16:297--321, 1965.
	
	\bibitem{ribenboim-length}
	Paulo Ribenboim.
	\newblock Valuations and lengths of constructible modules.
	\newblock {\em J. Reine Angew. Math.}, 283/284:186--201, 1976.
	
	\bibitem{length-entropy}
	Luigi Salce, Peter V\'amos, and Simone Virili.
	\newblock Length functions, multiplicities and algebraic entropy.
	\newblock {\em Forum Math.}, 25(2):255--282, 2013.
	
	\bibitem{length-entropy-2}
	Luigi Salce and Simone Virili.
	\newblock The addition theorem for algebraic entropies induced by non-discrete
	length functions.
	\newblock {\em Forum Math.}, 28(6):1143--1157, 2016.
	
	\bibitem{salce-zanardo-forum}
	Luigi Salce and Paolo Zanardo.
	\newblock A general notion of algebraic entropy and the rank-entropy.
	\newblock {\em Forum Math.}, 21(4):579--599, 2009.
	
	\bibitem{stable_prufer}
	Dario Spirito.
	\newblock Towards a classification of stable semistar operations on a
	{P}r\"ufer domain.
	\newblock {\em Comm. Algebra}, 46(4):1831--1842, 2018.
	
	\bibitem{starloc}
	Dario Spirito.
	\newblock Jaffard families and localizations of star operations.
	\newblock {\em J. Commut. Algebra}, to appear.
	
	\bibitem{vamos-additive}
	Peter V\'amos.
	\newblock Additive functions and duality over {N}oetherian rings.
	\newblock {\em Quart. J. Math. Oxford Ser. (2)}, 19:43--55, 1968.
	
	\bibitem{zanardo_length}
	Paolo Zanardo.
	\newblock Multiplicative invariants and length functions over valuation
	domains.
	\newblock {\em J. Commut. Algebra}, 3(4):561--587, 2011.
	
\end{thebibliography}
\end{document}